\documentclass[11pt]{amsart}

\usepackage{amsthm}
\usepackage{amsmath}
\usepackage{amssymb}
\usepackage{mathtools}
\usepackage{xcolor}
\usepackage[h]{esvect}
\usepackage[noadjust]{cite}
\definecolor{darkblue}{rgb}{0,0,0.5}
\usepackage[
bookmarksopen=true,
bookmarksopenlevel=1,
colorlinks=true,
linkcolor=darkblue,
linktoc=page,
citecolor=darkblue,
]{hyperref}

\usepackage{verbatim}
\usepackage{geometry}
\hypersetup{
    pdftitle={A special case of Vu's conjecture},
    pdfauthor={Kelly, K\"uhn, and Osthus}
}
\usepackage[font={small}]{caption}
\numberwithin{equation}{section}

\usepackage{fancyhdr}

\usepackage[shortlabels, inline]{enumitem}
\setlist{nolistsep, noitemsep}

\newtheorem{theorem}{Theorem}[section]

\newtheorem{lemma}[theorem]{Lemma}

\newtheorem{corollary}[theorem]{Corollary}
\newtheorem{conjecture}[theorem]{Conjecture}
\theoremstyle{definition}
\newtheorem{definition}[theorem]{Definition}

\usepackage{thmtools}
\usepackage{thm-restate}

\declaretheorem[sibling=theorem, name=Proposition]{proposition}

\newcommand{\Prob}[1]{\ensuremath{%
    \mathbb P\left[#1\right]
}}

\newcommand{\Expect}[1]{\ensuremath{%
    \mathbb E\left[#1\right]
  }}

\newcommand{\eps}{\ensuremath{\varepsilon}}

\newcommand{\cH}{\ensuremath{\mathcal H}}
\newcommand{\cG}{\ensuremath{\mathcal G}}
\newcommand{\X}{\ensuremath{\mathbf X}}

\def\COMMENT#1{}

\def\tom#1{}

\newcommand{\APPENDIX}[1]{}
\newcommand{\NOTAPPENDIX}[1]{#1}
\renewcommand{\APPENDIX}[1]{#1}                    
\renewcommand{\NOTAPPENDIX}[1]{}                   

\newcommand{\remainingColoursRV}{\ensuremath \mathbf{\Lambda}}
\newcommand{\uncolActivatedNbrsRV}{\ensuremath \mathbf{Y}}

\newcommand{\colDegRV}{\ensuremath \mathbf{D}}
\newcommand{\keptUnactColsRV}{\ensuremath \mathbf{R}}

\newcommand{\sampleSpace}{\ensuremath \mathbf\Omega}
\newcommand{\sigmaAlg}{\ensuremath \mathbf\Sigma}
\newcommand{\conflictsRV}{\ensuremath \mathbf{F}}
\newcommand{\change}{\ensuremath \delta}

\newcommand{\newlambda}{\ensuremath{\Lambda_\mathbb E}}
\newcommand{\newdeg}{\ensuremath{D_{\mathbb E}}}

\newcommand{\row}{\ensuremath{A}}
\newcommand{\col}{\ensuremath{B}}
\newcommand{\sym}{\ensuremath{C}}

\begin{document}

\title[A special case of Vu's conjecture]{
  A special case of Vu's conjecture: Coloring nearly disjoint graphs of bounded maximum degree
}

\date{}

\author[Kelly]{Tom Kelly}
\address{School of Mathematics, Georgia Institute of Technology, Atlanta, GA, USA}
\email{tom.kelly@gatech.edu}

\author[K\"uhn]{Daniela K\"uhn}

\author[Osthus]{Deryk Osthus}

\address{School of Mathematics, University of Birmingham,
  Edgbaston, Birmingham, B15 2TT, United Kingdom}
\email{\{D.Kuhn, D.Osthus\}@bham.ac.uk}

\thanks{This project has received partial funding from the European Research
Council (ERC) under the European Union's Horizon 2020 research and innovation programme (grant agreement no.~786198, T.~Kelly, D.~K\"uhn and D.~Osthus).
The research leading to these results was also partially supported by the EPSRC, grant nos.~EP/N019504/1 (T.~Kelly and D.~K\"uhn) and EP/S00100X/1 (D.~Osthus).}

\begin{abstract}
  A collection of graphs is \textit{nearly disjoint} if every pair of them intersects in at most one vertex.  We prove that if $G_1, \dots, G_m$ are nearly disjoint graphs of maximum degree at most $D$, then the following holds.  For every fixed $C$, if each vertex $v \in \bigcup_{i=1}^m V(G_i)$ is contained in at most $C$ of the graphs $G_1, \dots, G_m$, then the (list) chromatic number of $\bigcup_{i=1}^m G_i$ is at most $D + o(D)$.  This result confirms a special case of a conjecture of Vu and generalizes Kahn's bound on the list chromatic index of linear uniform hypergraphs of bounded maximum degree.  In fact, this result holds for the correspondence (or DP) chromatic number and thus implies a recent result of Molloy and Postle, and we derive this result from a more general list coloring result in the setting of `color degrees' that also implies a result of Reed and Sudakov.  
\end{abstract}

\maketitle

\section{Introduction}

A \textit{proper coloring} of a graph $G$ is an assignment of colors to the vertices of $G$ such that no two vertices assigned the same color are adjacent, and the \textit{chromatic number} of $G$, denoted $\chi(G)$, is the minimal number of colors needed to properly color $G$.  A \textit{list assignment} for a graph $G$ is a map $L$ whose domain is the vertex set $V(G)$ of $G$ where $L(v)$ is a finite set called the `list of available colors for $v$', and an \textit{$L$-coloring} of $G$ is a proper coloring $\phi$ of $G$ such that $\phi(v) \in L(v)$ for every $v\in V(G)$.  If $G$ has an $L$-coloring, then $G$ is \textit{$L$-colorable}.  The \textit{list chromatic number} of $G$, denoted $\chi_\ell(G)$, is the minimum $k \in \mathbb N$ such that $G$ is $L$-colorable for every list assignment $L$ satisfying $|L(v)| \geq k$ for all $v \in V(G)$.

We say graphs $G_1, \dots, G_m$ are \textit{nearly disjoint} if $|V(G_i) \cap V(G_j)| \leq 1$ for all distinct $i, j \in [m]$.  In this paper, we prove an asymptotically optimal bound on the (list) chromatic number of a union of nearly disjoint graphs of bounded maximum degree (see Theorem~\ref{cor:main}), which confirms a special case of a conjecture of Vu (see Conjecture~\ref{conj:vu}) and generalizes several well-known results on hypergraph edge-coloring.  We derive our bound from a more general result (see Theorem~\ref{thm:main}) concerning $L$-colorings when $L$ is a list assignment for a union of nearly disjoint graphs such that $L$ has bounded maximum \textit{color degree}.  This result also implies a result of Reed and Sudakov~\cite{ReSu02}.  We also prove a bound on the chromatic number of a union of nearly disjoint graphs of bounded chromatic number (see Theorem~\ref{thm:chi-efl}).

\subsection{Line graphs of linear hypergraphs and nearly disjoint graph unions}\label{sec:nd-union-bounded-max-deg}

A \textit{hypergraph} $\cH$ is a pair $\cH = (V, E)$ where $V$ is a set whose elements are called \textit{vertices} and $E \subseteq 2^V$ is a set of subsets of $V$ whose elements are called \textit{edges}.  A \textit{proper edge-coloring} of a hypergraph $\cH$ is an assignment of colors to the edges of $\cH$ such that no two edges of the same color share a vertex, and the \textit{chromatic index} of $\cH$, denoted $\chi'(\cH)$, is the minimum number of colors used by a proper edge-coloring of $\cH$.  The \textit{line graph} of a hypergraph $\cH$ is the graph $G$ where $V(G)$ is the edge set of $\cH$, and $e, f \in V(G)$ are adjacent in $G$ if $e \cap f \neq \emptyset$.  Note that the chromatic index of a hypergraph is the chromatic number of its line graph.  The \textit{list chromatic index} of a hypergraph $\cH$, denoted $\chi'_\ell(\cH)$, is the list chromatic number of its line graph.

A hypergraph $\cH$ is \textit{linear} if every two distinct edges of $\cH$ intersect in at most one vertex and \textit{$k$-bounded} if every edge of $\cH$ has size at most $k$.  A celebrated result of Pippenger and Spencer~\cite{PS89} implies that $k$-bounded linear hypergraphs of maximum degree at most $D$ have chromatic index at most $D + o(D)$ for fixed $k$ as $D\rightarrow\infty$, and a similarly influential result of Kahn~\cite{K96} generalizes the Pippenger--Spencer theorem to list coloring.  (Both of these results apply more generally to hypergraphs of small codegree, and not only to linear hypergraphs.)  An intermediate result of Kahn~\cite[Theorem 3]{K92} was also crucial to proving the Erd\H os--Faber--Lov\' asz conjecture asymptotically.  This conjecture states that a nearly disjoint union of $n$ complete graphs, each on at most $n$ vertices, has chromatic number at most $n$, and it was recently confirmed for all large $n$ by Kang, Methuku, and the authors~\cite{KKKMO21}.  Note that the line graph of a $k$-bounded linear hypergraph is a union of nearly disjoint complete graphs, such that at most $k$ of them contain any given vertex.  Thus, the following result strengthens both the Pippenger--Spencer theorem and Kahn's list edge-coloring theorem for the case of linear hypergraphs.

\begin{theorem}\label{cor:main}
  For every $C, \eps > 0$, the following holds for all sufficiently large $D$.  If $G_1, \dots, G_m$ are nearly disjoint graphs of maximum degree at most $D$, such that each vertex $v \in \bigcup_{i=1}^m V(G_i)$ is contained in at most $C$ of them, then $\chi_\ell(\bigcup_{i=1}^m G_i) \leq (1 + \eps)D$.  
\end{theorem}

Theorem~\ref{cor:main} also confirms a special case of the following conjecture of Vu~\cite{Vu02}, and (as our discussion shows) it recovers several of its significant consequences.

\begin{conjecture}[Vu~\cite{Vu02}]\label{conj:vu}
  For every $\zeta, \eps > 0$, the following holds for all sufficiently large $D$. If $G$ is a graph of maximum degree at most $D$ and every two distinct vertices have at most $\zeta D$ common neighbors in $G$, then $\chi_\ell (G) \leq (\zeta + \eps) D$.
\end{conjecture}
Indeed, Conjecture~\ref{conj:vu}, if true, implies Theorem~\ref{cor:main} with $1 / C$ and $C(D + C^2)$ playing the roles of $\zeta$ and $D$, respectively.  If $G_1, \dots, G_m$ are nearly disjoint graphs of maximum degree at most $D$, such that each vertex $v \in \bigcup_{i=1}^m V(G_i)$ is contained in at most $C$ of them, then every two distinct vertices in $G\coloneqq \bigcup_{i=1}^m G_i$ have at most $D + C^2$ common neighbors\COMMENT{Let $u, v \in V(G)$ be distinct vertices, and let $i \in [m]$ such that $|\{u, v\} \cap V(G_i)|$ is maximum.  Now $u$ and $v$ have at most $D$ common neighbors in $G_i$.  By the choice of $i$, if $u \in V(G_j)$ for $j \in [m]\setminus \{i\}$, then $v \notin V(G_i)$, and thus, $u$ and $v$ have at most $C$ common neighbors in $G_j$ (as for each such $z \in V(G_j)$ there must be a different index $k$ such that $v, z \in V(G_k)$).  Since there are at most $C$ such $j \in [m]\setminus\{i\}$, $u$ and $v$ have at most $D + C^2$ common neighbors.}, and  $G$ has maximum degree at most $CD$.  Vu~\cite{Vu02} initially observed that a similar argument shows that Conjecture~\ref{conj:vu} implies Kahn's~\cite{K96} bound on the list chromatic index of linear hypergraphs of bounded maximum degree, and this was a major motivation for the conjecture: `\textit{The bound in [Conjecture~\ref{conj:vu}], if valid, would be an \textit{amazing} result.  For instance, it would immediately imply a deep theorem of Kahn on the list chromatic index of [linear] hypergraphs.}'~\cite[pg.~109]{Vu02}.
The only other nontrivial result towards Vu's conjecture so far was recently obtained by Hurley, de Joannis de Verclos, and Kang~\cite{HdVK20}, who proved a bound on the chromatic number of graphs as in Conjecture~\ref{conj:vu}.  Their result confirms the weaker version of the conjecture that considers the chromatic number rather than the list chromatic number, in the case when $\zeta > 1 - o(\eps^{2/3})$.  We remark that even finding an independent set of the required size is still an open problem (which was already raised by Vu~\cite{Vu02}).

\subsection{Color degrees}

Our main result in this paper is actually a generalization of Theorem~\ref{cor:main} that also implies a result of Reed and Sudakov~\cite{ReSu02}.  
For a graph $G$ with list assignment $L$, we define the \textit{color degree} of each $v\in V(G)$ and $c \in L(v)$ to be $d_{G, L}(v, c) \coloneqq |\{u \in N_G(v) : c\in L(u)\}|$ and the \textit{maximum color degree} of $G$ and $L$ to be $\Delta(G, L) \coloneqq \max_{v\in V(G)}\max_{c\in L(v)}d_{G, L}(v, c)$.  Note that $\Delta(G, L) \leq \Delta(G)$.  The following is our main result.
\begin{theorem}\label{thm:main}
  For every $C, \eps > 0$, the following holds for all sufficiently large $D$.  Let $G_1, \dots, G_m$ be graphs that 
  \begin{enumerate}[(i)]
  \item\label{hypo:intersection-bound} are nearly disjoint and 
  \item\label{hypo:containment-bound} satisfy $|\{i \in [m] : v \in V(G_i)\}| \leq C$ for every $v \in \bigcup_{i=1}^m V(G_i)$.
  \end{enumerate}
 If $L$ is a list assignment for $G \coloneqq \bigcup_{i=1}^m G_i$ satisfying
  \begin{enumerate}[(i)]\stepcounter{enumi}\stepcounter{enumi}
  \item\label{hypo:degree-bound} $\Delta(G_i, L|_{V(G_i)}) \leq D$ for every $i \in [m]$ and
  \item\label{hypo:list-bound} $|L(v)| \geq (1 + \eps)D$ for every $v\in V(G)$,
  \end{enumerate}
  then $G$ is $L$-colorable.
\end{theorem}

Note that Theorem~\ref{cor:main} immediately follows from Theorem~\ref{thm:main}.

Theorem~\ref{thm:main} actually holds more generally for \textit{correspondence coloring}, also known as \textit{DP-coloring} (see Theorem~\ref{thm:main-corr}), and this result also implies the recent result of Molloy and Postle~\cite{M18}, that $k$-bounded linear hypergraphs of maximum degree at most $D$ have \textit{correspondence chromatic index} at most $D + o(D)$.  A proof of Theorem~\ref{thm:main-corr} can be obtained with only minor modifications to our argument used to prove Theorem~\ref{thm:main}.  Thus, for the sake of presentation, we choose to first provide a complete proof of Theorem~\ref{thm:main} for list coloring.  Then, in Section~\ref{section:corr-col}, we describe the modifications necessary to prove Theorem~\ref{thm:main-corr}, its correspondence coloring generalization.

The case $m = 1$ (or equivalently, $C = 1$) in Theorem~\ref{thm:main} is a result of Reed and Sudakov~\cite{ReSu02}.  For every $C \in \mathbb N$, let $f_C(D)$ denote the smallest integer for which the following holds: If $G_1, \dots, G_m$ are graphs satisfying \ref{hypo:intersection-bound} and \ref{hypo:containment-bound}, and if $L$ is a list assignment for $G \coloneqq \bigcup_{i=1}^m G_i$ satisfying \ref{hypo:degree-bound} such that $|L(v)| \geq D + f_C(D)$ for every $v \in V(G)$, then $G$ is $L$-colorable.  Clearly, $f_C(D) \geq 1$\COMMENT{for example, consider when $m = 1$, $G_1\cong K_{D+1}$, and every vertex has the same list}.  Theorem~\ref{thm:main} implies that $f_C(D) = o(D)$, but it would be interesting to prove better asymptotics.  A bound of $f_C(D) \leq D^{1 - 1/C}$ for $C \geq 2$ would match the best known bound on the list chromatic index of linear hypergraphs due to Molloy and Reed~\cite{MR2000}.  Reed~\cite{Ree99} conjectured that $f_1(D) = 1$ for every $D \in \mathbb N$.  Bohman and Holzman~\cite{BoHo02} disproved this conjecture, and Reed and Sudakov~\cite{ReSu02} asked whether $f_1(D) = O(1)$.  Another interesting direction would be to generalize Theorem~\ref{thm:main} so as to also imply the generalization of the Reed--Sudakov result in which the maximum color degree condition is replaced with an average one, proved by Glock and Sudakov~\cite{GS20} and Kang and Kelly~\cite{KK20}.  For other related results and open problems involving list coloring and color degrees, see e.g.~\cite{AA20, ABD21, ABD23, CK20}.  More generally, for a recent survey on coloring results and open problems obtained via nibble methods, see~\cite{KKKMOsurvey}.

\subsection{Coloring nearly disjoint graphs of bounded chromatic number}

Erd\H{o}s proposed several variations of the Erd\H{o}s--Faber--Lov\'{a}sz conjecture.  For example, relaxing the condition of nearly disjointness, Erd\H{o}s~\cite[Problem 9]{erdos1979} asked for the largest possible chromatic number of a union of $n$ complete graphs, each on at most $n$ vertices, that pairwise intersect in at most $t$ vertices.  Building on the methods of \cite{KKKMO21}, this problem was recently solved by Kang, Methuku, and the authors~\cite{KKKMO21t-EFL}.  Here we discuss a problem related to a question of Erd\H os~\cite[p.~26]{erdos1981} on bounds on the chromatic number of a union of nearly disjoint graphs $G_i$ if we know their chromatic number (rather than their maximum degree as in Section~\ref{sec:nd-union-bounded-max-deg}).

Accordingly, for a family $\cG$ of graphs and $m \in \mathbb N$, let $f(m, \cG)$ be the largest possible chromatic number of the union of at most $m$ nearly disjoint graphs in $\cG$.  
Recall that the Erd\H os--Faber--Lov\' asz conjecture states that a nearly disjoint union of $n$ complete graphs, each on at most $n$ vertices, has chromatic number at most $n$.  Thus, the Erd\H os--Faber--Lov\' asz conjecture can be expressed as follows: For all $n \in \mathbb N$, we have $f(n, \{K_1, \dots, K_n\}) \leq n$, where $K_t$ denotes the complete graph on $t$ vertices.  However, it is straightforward to show that $f(n, \{K_1, \dots, K_n\}) \leq \max\{n, f(n, \{K_{n-1}\})\}$\COMMENT{since we can add vertices to the complete graphs with less than $n - 1$ vertices and remove (and color last) a vertex $x$ of degree at most $n - 1$ from every complete graph with $n$ vertices (where $x$ lies in only one of the complete graphs)}, and it is well-known that $f(n, \{K_{n-1}\}) \geq n$ for $n \geq 3$\COMMENT{the line graph of the $n$-vertex degenerate plane is an $n$-vertex complete graph, which is the nearly disjoint union of one complete graph on $n - 1$ vertices and $n - 1$ complete graphs on two vertices}.  Hence, the Erd\H os--Faber--Lov\' asz conjecture can be reduced to the seemingly weaker statement that $f(n, \{K_{n-1}\}) \leq n$ for every $n \in \mathbb N$.

Erd\H os~\cite[p.~26]{erdos1981} proposed the following variation of the Erd\H os--Faber--Lov\' asz conjecture: `\textit{Let $G_1, \dots, G_m$ be $m$ graphs each of chromatic number $n$.  Assume that no two $G$'s have an edge in common.  What is the smallest $m$ for which $\bigcup_{i=1}^m G_i$ has chromatic number greater than $n$?  Perhaps one can further demand that any two $G$'s have at most one vertex in common.}'  It turns out the answers to these two questions are `two' and `three' respectively.  (For the former, observe that a complete graph on $n + 1$ vertices can be expressed as the union of a complete graph on $n$ vertices and a star with no edges in common; for the latter, see Theorem~\ref{thm:chi-efl}\ref{chi-efl-lower} below.)  The latter question can also be expressed as follows: What is the smallest $m$ for which $f(m, \cG^\chi_n) > n$, where $\cG^\chi_n$ is the set of graphs of chromatic number at most $n$?  Although this question is straightforward to answer, we believe that further analysis of the function $f(m, \cG^\chi_n)$ itself is warranted (which was probably the original intention of the above question of Erd\H os).  In this direction, we prove the following.

\begin{theorem}\label{thm:chi-efl}~
  \begin{enumerate}[(i)]
  \item\label{chi-efl-upper} The following holds for all $m, n \in \mathbb N$ with $m + n$ sufficiently large: If $G_1, \dots, G_m$ are nearly disjoint graphs, each of chromatic number at most $n$, then $\chi(\bigcup_{i=1}^m G_i) \leq m + n - 2$.
  \item\label{chi-efl-lower} For every $n \geq 2$, there exist nearly disjoint graphs $G_1, G_2, G_3$, each of chromatic number at most $n$, such that $\chi(G_1\cup G_2 \cup G_3) \geq n + 1$.
  \end{enumerate}
\end{theorem}
Theorem~\ref{thm:chi-efl}\ref{chi-efl-upper} implies that $f(m, \cG^\chi_n) \leq m + n - 2$ when $m + n$ is sufficiently large, and Theorem~\ref{thm:chi-efl}\ref{chi-efl-lower} implies that this bound is tight for $m = 3$.  However, this bound can likely be improved for larger $m$.  In particular, it is tempting to conjecture that $f(n, \cG^\chi_{n-1}) \leq n$, which by the discussion above, would imply the Erd\H os--Faber--Lov\' asz conjecture if true; however, this was disproved by Postle~\cite{Ppc}, as follows.

\begin{theorem}[Postle~\cite{Ppc}]\label{thm:postle}
  For every $m \in \mathbb N$ divisible by three and $n \geq m - 1$, there exist nearly disjoint graphs $G_1, \dots, G_m$, each of chromatic number at most $n$, such that $\chi(\bigcup_{i=1}^m G_i) \geq n + m / 6$.
\end{theorem}

Because of the connection to the Erd\H{o}s--Faber--Lov\'{a}sz conjecture, we find it most interesting to study the function $f(n, \cG^\chi_{n-1})$.  Theorems~\ref{thm:chi-efl}\ref{chi-efl-upper} and \ref{thm:postle}\COMMENT{applied with $m \in [n - 2, n]$ divisible by three and $n - 1$ playing the role of $n$} imply that for large $n$ we have
\begin{equation}\label{eqn:chi-efl-bounds}
  \frac{7n - 8}{6}\COMMENT{$=(n-1) + (n - 2)/6 \leq n - 1 + m / 6 \leq f(m, \cG^\chi_{n - 1})$} \leq f(n, \cG^\chi_{n - 1}) \leq 2n - 3.
\end{equation}
We think it would be interesting to determine $\lim_{n\rightarrow\infty}f(n, \cG^\chi_{n-1}) / n$, assuming the limit exists.  By \eqref{eqn:chi-efl-bounds}, it would be in the range $[7/6, 2]$.

Considering Theorem~\ref{cor:main}, it is natural to also ask for $f(n, \cG^\Delta_{n - 1})$ (or $f(n, \cG^\Delta_{n - 2})$), where $\cG^\Delta_{n}$ is the set of graphs of maximum degree at most $n$.  However, it is straightforward to show that $f(n, \cG^\Delta_{n - 1}) = f(n, \{K_n\}) = f(n, \{K_{n - 1}\}) = f(n, \cG^\Delta_{n - 2})$ for every $n \geq 3$.\COMMENT{A nearly disjoint union of $n$ graphs $G_1, \dots, G_n$, each of maximum degree at most $n - 1$, has a vertex of degree at most $n - 1$ unless $|V(G_i)| \leq n - 1$ for every $i \in [n]$.  Thus, $f(n, \cG^\Delta_{n - 1}) \leq \max\{n, f(n, \{K_{n-1}\})\} \leq f(n, \{K_n\})$, and $f(n, \cG^\Delta_{n - 1}) \geq f(n, \{K_n\})$ since $K_n \in \cG^\Delta_{n - 1}$.  Moreover, a similar argument shows the last equality, and the argument above implies the second equality.}

\subsection{Outline of the paper}
Sections~\ref{section:main-proof}--\ref{section:nibble} are devoted to the proof of Theorem~\ref{thm:main}.  We prove Theorem~\ref{thm:main} using a semi-random coloring procedure (also referred to as the `nibble method').  Each step of this procedure is obtained by an application of Lemma~\ref{lemma:nibble}, which is proved in Section~\ref{section:nibble}.  Section~\ref{section:tools} provides some probabilistic tools needed in the proof of Lemma~\ref{lemma:nibble}.  In Section~\ref{section:main-proof}, we prove Theorem~\ref{thm:main} (assuming Lemma~\ref{lemma:nibble}).
As mentioned, Theorem~\ref{thm:main} holds more generally for correspondence coloring.  In Section~\ref{section:corr-col}, we formally state this generalization and describe how to modify the proof of Theorem~\ref{thm:main} to prove it.
In Section~\ref{section:chi-efl}, we prove Theorems~\ref{thm:chi-efl} and \ref{thm:postle}.

\section{Proof of Theorem~\ref{thm:main}}\label{section:main-proof}

In this section we prove Theorem~\ref{thm:main}, assuming the following lemma.  Given $k \in \mathbb Z$, we write $\log^k D \coloneqq (\log D)^k$, where the logarithm is base $e$.  Given $\Lambda, D, C, p \in \mathbb R$, we also define
  \begin{align*}
    \newlambda(\Lambda, D, C, p) &\coloneqq \left(1 - \frac{p}{\Lambda}\right)^{DC}\Lambda && \mathrm{and}\\
    \newdeg(\Lambda, D, C, p) &\coloneqq \left((1 - p)\left(1 - \frac{p}{\Lambda}\right)^{D(C-1)} + p\left(1 - \left(1 - \frac{p}{\Lambda}\right)^{DC}\right)\right) D.
  \end{align*}

\begin{lemma}\label{lemma:nibble}
  For every $C, \eps > 0$, there exists $D_{\ref{lemma:nibble}}$ such that the following holds for every $D \geq D_{\ref{lemma:nibble}}$.
  Let $G_1, \dots, G_m$ be graphs, and let $L$ be a list assignment for $G \coloneqq \bigcup_{i=1}^m G_i$ satisfying \ref{hypo:intersection-bound}--\ref{hypo:list-bound} (of Theorem~\ref{thm:main}).
  If $|L(v)| = \lceil\Lambda\rceil$ for every $v\in V(G)$, where $(1 + \eps)D \le \Lambda \le 10CD$,  and if $\log^{-1} D \geq p \geq \log^{-2}D$,
  then there exist $X\subseteq V(G)$, an $L|_X$-coloring $\phi$ of $G[X]$, and a list assignment $L'$ for $G - X$ satisfying $L'(v) \subseteq L(v) \setminus \{\phi(u) : u \in N_G(v)\cap X\}$ for every $v\in V(G)\setminus X$,
  such that 
  \begin{enumerate}[(\ref*{lemma:nibble}.1)]
  \item\label{nibble:list-size} $|L'(v)| = \left\lceil \newlambda(\Lambda, D, C, p) - \Lambda^{4/5}\right\rceil$ for every $v\in V(G)\setminus X$ and
  \item\label{nibble:color-degree} $\Delta(G_i - X, L'|_{V(G_i - X)}) \leq \newdeg(\Lambda, D, C, p) + D^{4/5}$ for every $i \in [m]$.
  \end{enumerate}
\end{lemma}

We prove Lemma~\ref{lemma:nibble} in Section~\ref{section:nibble} by analyzing a random coloring procedure, but we already explain some of the ideas involved in the proof now.
In our random coloring procedure, we randomly assign each vertex $v \in V(G)$ a color $\psi(v) \in L(v)$ uniformly at random from its list, and we `activate' vertices independently at random with probability $p$.  (We really only need to assign colors to activated vertices, but for technical reasons it is convenient to define the procedure this way.)
For each vertex, we remove any color from its list assigned to an activated neighbor, and we let $X$ be the set of activated vertices whose color was not assigned to any of its activated neighbors.  We obtain $L'$ by further truncating each vertex's list to have the desired size, and we show that with nonzero probability, $X$, $L'$, and $\phi \coloneqq \psi|_X$ satisfy \ref{nibble:list-size} and \ref{nibble:color-degree}.

To that end, assuming without loss of generality that $d_{G_i, L}(v, c) = D$ for every $i \in [m]$ and $v \in V(G_i)$ with $c \in L(v)$ (see Proposition~\ref{prop:embedding}), it is straightforward to show that a vertex `keeps' a given color in its list with probability $(1 - p / \Lambda)^{DC}$ (see Proposition~\ref{prop:keep}).
Thus, after applying this procedure, $\newlambda(\Lambda, D, C, p)$ is the expected number of remaining available colors for each vertex, and $\newdeg(\Lambda, D, C, p)$ is an upper bound on the expected color degree of each pair of vertex and color from its list (see \eqref{eqn:degree-overcount} and Lemma~\ref{lemma:expectations}).  We show that with very high probability (at least $1 - \exp\left(-D^{1/4}\right)$ -- see Lemma~\ref{lemma:concentrations}), the number of remaining colors available for a given vertex, and the color degree in $G - X$ of a given pair of vertex and color, are not significantly larger than $\newlambda(\Lambda, D, C, p)$ and $\newdeg(\Lambda, D, C, p)$, respectively.  We complete the proof by applying the Lov\'{a}sz Local Lemma.

To prove Theorem~\ref{thm:main}, we iteratively apply Lemma~\ref{lemma:nibble} $O(C\eps^{-1}\log D)$ times, each time reapplying the lemma with $p = \log^{-1}D$ and with $G_1 - X, \dots, G_m - X$ and $L'$ playing the roles of $G_1, \dots, G_m$ and $L$, respectively, before completing the coloring with Lemma~\ref{lemma:finishing-blow} below.  The requirement $L'(v) \subseteq L(v) \setminus \{\phi(u) : u \in N_G(v) \cap X\}$ ensures that adjacent vertices are not assigned the same color in different iterations, so each iteration extends a proper partial coloring of $G$.

 This approach -- known as the `nibble method' or the `semi-random method' -- has led to numerous important developments in graph coloring (see \cite{MR02, KKKMOsurvey}).
 Our random coloring procedure was in fact already used by Reed and Sudakov~\cite{ReSu02} to prove the special case of Theorem~\ref{thm:main} when $C = 1$; however, our analysis of the procedure is different even in this special case, drawing ideas from a recent proof of Kang and Kelly~\cite{KK20}, and we need new ideas to analyze the procedure for $C > 1$.
 Recall that Theorem~\ref{thm:main} also implies Kahn's~\cite{K96} bound on the chromatic index of linear uniform hypergraphs and that Theorem~\ref{thm:main-corr} implies Molloy and Postle's~\cite{M18} generalization of Kahn's result to correspondence coloring.  These edge-coloring results are also proved with a semi-random approach, but the random coloring procedures used in these proofs are slightly different from ours.
Since vertices may lose colors unnecessarily in our procedure---it would suffice to only remove colors from a vertex's list that were assigned to a neighbor in $X$, but this version of the procedure would be more challenging to analyze---Reed and Sudakov~\cite{ReSu02} called the procedure `wasteful' (see also \cite[Chapter 12]{MR02}).  Nevertheless, this wastefulness is negligible because our activation probability $p$ is small.    However, Kahn~\cite{K96} did not consider activation probabilities and consequently could not afford to use a wasteful variant of his procedure.  Molloy and Postle's~\cite{M18} proof does consider activation probabilities but is still different from ours, even in the special case of edge coloring; in particular, vertices (of the line graph) may be assigned multiple colors in their coloring procedure.

Let us now explain how we use Lemma~\ref{lemma:nibble} to prove Theorem~\ref{thm:main}.
Crucially, \ref{nibble:list-size} and \ref{nibble:color-degree} together imply that after each iteration of Lemma~\ref{lemma:nibble}, the ratio of the number of remaining available colors for each vertex to the maximum remaining color degree in each $G_i$, while initially only $1 + \eps$, improves by a factor of at least $1 + \eps p / 4$.  Moreover, the number of available colors does not decrease too much in each iteration.  The following proposition makes this calculation precise.

\begin{proposition}\label{prop:calculations}
  For every $C\geq1$ and $0 < \eps < 1$ there exists $D_{\ref{prop:calculations}}$ such that the following holds for every $D \geq D_{\ref{prop:calculations}}$.
  If $\log^{-1}D \geq p \geq \log^{-2}D$ and $10CD \geq \Lambda \geq (1 + \eps)D$, then
  \begin{equation}\label{eqn:ratio-improvement}
    \frac{\newlambda(\Lambda, D, C, p) - \Lambda^{4/5}}{\newdeg(\Lambda, D, C, p) + D^{4/5}} \geq (1 + \eps p / 4)\frac{\Lambda}{D}
  \end{equation}
  and
  \begin{equation}\label{eqn:degree-maintained}
    \newdeg(\Lambda, D, C, p) \geq (1 - pC)D.
  \end{equation}
\end{proposition}

\begin{proof}
  First we prove \eqref{eqn:degree-maintained}.
  Clearly,\COMMENT{By the well-known inequality $(1 + x / n)^n \geq 1 + x$ for $|x| \leq n$, and using that $\Lambda \geq D$, we have
  \begin{equation*}
    \left(1 - \frac{p}{\Lambda}\right)^{D(C-1)} \geq 1 - \frac{pD(C - 1)}{\Lambda} \geq 1 - p(C - 1).
  \end{equation*}}
  \begin{equation*}
    (1 - p)\left(1 - \frac{p}{\Lambda}\right)^{D(C-1)}D \geq (1 - p)(1 - p(C - 1))D \geq (1 - pC)D,
  \end{equation*}
  and \eqref{eqn:degree-maintained} follows immediately, since $p\left(1 - \left(1 - p / \Lambda\right)^{DC}\right) > 0$.

  Now we prove \eqref{eqn:ratio-improvement}.  By \eqref{eqn:degree-maintained}, since $\Lambda \leq 10CD$ and $\log^{-1} D \geq p \geq \log^{-2}D$,
  \begin{equation*}
    \frac{\Lambda^{4/5}}{\newdeg(\Lambda, D, C, p) + D^{4/5}} \leq \frac{10CD^{4/5}}{(1 - pC)D} \leq \frac{\eps p}{100}\cdot\frac{\Lambda}{D}.
  \end{equation*}
  Let
  \begin{equation*}
    W \coloneqq (1 - p)\left(1 - \frac{p}{\Lambda}\right)^{-D} + p\left(\left(1 - \frac{p}{\Lambda}\right)^{-DC} - 1\right) + \left(1 - \frac{p}{\Lambda}\right)^{-DC}D^{-1/5},
  \end{equation*}
  and note that 
  \begin{equation*}
    \frac{\newlambda(\Lambda, D, C, p)}{\newdeg(\Lambda, D, C, p) + D^{4/5}}
    =  \frac{1}{W}\cdot \frac{\Lambda}{D}.
  \end{equation*}
  Since $1 + x \leq (1 + x / n)^n$ for $x \leq |n|$ and $\Lambda \geq (1 + \eps)D$, we have
  \begin{equation*}
    W  \leq \frac{1 - p}{1 - p/(1 + \eps)} + p\left(\frac{1}{1 - pC/(1 + \eps)} - 1\right) + \frac{D^{-1/5}}{1 - pC/(1 + \eps)}\COMMENT{$\leq (1 - p)\left(1 + \frac{p}{1 + \eps} + \frac{2Cp^2}{(1 + \eps)^2}\right) + pC\left(\frac{p}{1 + \eps} + \frac{2p^2}{(1 + \eps)^2}\right) + 2D^{-1/5}$}
    \leq 1 - p + \frac{p}{1 + \eps} +O(p^2) \leq\COMMENT{since $\eps < 1$} 1 - \frac{p\eps}{3}.
  \end{equation*}
  Since $(1 - p\eps / 3)^{-1} \geq 1 + p\eps / 3$, by combining the inequalities above, we have that the left side of \eqref{eqn:ratio-improvement} is at least
  \begin{equation*}
    \left(1 + p\eps / 3 - \eps p  /100\right)\frac{\Lambda}{D} \geq (1 + \eps p / 4)\frac{\Lambda}{D},
  \end{equation*}
  as desired.
\end{proof}

We apply Lemma~\ref{lemma:nibble} iteratively until the ratio of the number of remaining available colors to the maximum remaining color degree reaches $8CD$: \eqref{eqn:ratio-improvement} will imply that this ratio improves by a factor of at least $1 + \eps / (4\log D)$, so that we only need at most $33C\eps^{-1} \log D$ iterations, and \eqref{eqn:degree-maintained} will ensure that carrying out this many iterations is indeed possible.  After this process, we can finish with the following result of Reed~\cite{Ree99} (which is proved via a simple application of the Lov\' asz Local Lemma).

\begin{lemma}[Reed~\cite{Ree99}]\label{lemma:finishing-blow}
  Let $G$ be a graph with list assignment $L$.  If $|L(v)| \geq 8D$ for every $v \in V(G)$ and $\Delta(G, L) \leq D$, then $G$ is $L$-colorable.
\end{lemma}

Now we prove Theorem~\ref{thm:main}.

\begin{proof}[Proof of Theorem~\ref{thm:main}]
    Without loss of generality, we may assume that $\eps < 1$.\COMMENT{We also assume $C \geq 1$, or else \ref{hypo:containment-bound} implies that each $G_i$ is empty.}  Let $p \coloneqq \log^{-1} D$, $D_0 \coloneqq D$ and $\Lambda_0 \coloneqq (1 + \eps)D_0$, and for each integer $0\le i \le 33C/(\eps p)$, let
  \begin{align*}
    \Lambda_{i + 1} &\coloneqq \newlambda(\Lambda_i, D_i, C, p) - \Lambda_i^{4/5} && \mathrm{and}\\
    D_{i + 1} &\coloneqq \newdeg(\Lambda_i, D_i, C, p) + D_i^{4/5}.
  \end{align*}
  Applying Proposition~\ref{prop:calculations} inductively,
  for every integer $0\le i \le 33C/(\eps p)$, by \eqref{eqn:degree-maintained} we have
  \begin{equation*}
    D_{i} \ge \left(1 - pC\right)^{33C/(\eps p)}D_0 \geq e^{-33 C^2 / \eps}(1 - 33C^3p / \eps)D_0 \geq e^{-34 C^2 / \eps}D_0,
  \end{equation*}
  where the second inequality uses that $(1 + x / n)^n \geq e^x(1 - x^2 / n)$ for $|x| \leq n$\COMMENT{with $33C / (\eps p)$ playing the role of $n$ and $-33C^2 / \eps$ playing the role of $x$}.
  By \eqref{eqn:ratio-improvement},
  \begin{equation*}
    \frac{\Lambda_{i}}{D_{i}} \geq \left(1 + \frac{\eps p}{4}\right)\frac{\Lambda_{i-1}}{D_{i-1}}.
  \end{equation*}
  Moreover, we have
  \begin{equation*}
    \frac{\lceil \Lambda_{\lfloor 33C/(\eps p)\rfloor} \rceil}{D_{\lfloor 33C/(\eps p)\rfloor}}
    \geq \left(1 + \frac{\eps p}{4}\right)^{\lfloor 33C/(\eps p) \rfloor}\frac{\Lambda_0}{D_0}
    \geq 8C.
  \end{equation*}
  In particular, there exists an integer $0 < i^* \leq 33C / (\eps p)$ such that $\lceil\Lambda_{i^*}\rceil / D_{i^*} \geq 8C$ and
  $\lceil\Lambda_{i^* - 1}\rceil / D_{i^* - 1} < 8C$.  We may assume $D$ is sufficiently large so that $D_{i^*} \geq D_{\ref{lemma:nibble}}$.

  By \ref{hypo:list-bound}, we may assume without loss of generality that $|L(v)| = \lceil\Lambda_0\rceil$ for every $v \in V(G)$, since we can truncate each list until equality holds.  Now let $G_{j,0} \coloneqq G_j$ for each $j \in [m]$, let $H_0 \coloneqq G$, and let $L_0 \coloneqq L$. Due to the above calculations, inductively by Lemma~\ref{lemma:nibble}, for each integer $0\le i < i^*$, there is a set $X_i \subseteq V(H_i)$ and an $L_i|_{X_i}$-coloring $\phi_i$ of $H_i[X_i]$ and a list assignment $L_{i+1}$ for $H_{i+1}\coloneqq H_i - X_i$ satisfying $L_{i+1}(v) \subseteq L_i(v) \setminus \{\phi_i(u) : u \in N_{H_i}(v) \cap X_i\}$ such that
  \begin{itemize}
  \item $\Delta(G_{j, i + 1}, L_{i + 1}|_{V(G_{j, i+1})}) \leq D_{i + 1}$ for each $j \in [m]$ and
  \item $|L_{i + 1}(v)| = \lceil \Lambda_{i + 1}\rceil \ge (1+\eps) D_{i+1}$,
  \end{itemize}
  where $G_{j, i + 1} \coloneqq G_{j, i} - X_i$.  
  
  Let $X \coloneqq \bigcup_{i=1}^{i^*-1} X_i$, let $\phi(v) \coloneqq \phi_i(v)$ if $v \in X_i$ for some integer $0\le i < i^*$, let $G' \coloneqq H_{i^*}$, and let $L' \coloneqq L_{i^*}$.  By construction, $\phi$ is an $L|_X$-coloring of $G[X]$, and by the choice of $i^*$, we have
  \begin{itemize}
  \item $|L'(v)| \geq 8C D_{i^*}$ and
  \item $\Delta(G', L') \leq C\max\{\Delta(G_{j, i^*}, L_{i^*}|_{V(G_{j,i^*})}) : {j \in [m]}\} \leq CD_{i^*}$.
  \end{itemize}
  Therefore by Lemma~\ref{lemma:finishing-blow}, $G'$ has an $L'$-coloring $\phi'$, and we can combine $\phi$ and $\phi'$ to obtain an $L$-coloring of $G$, as desired.
\end{proof}

\section{Probabilistic tools}\label{section:tools}

In this section we provide some probabilistic tools used in the proof of Lemma~\ref{lemma:nibble} in Section~\ref{section:nibble}.  The first such tool is the Lov\'asz Local Lemma.

\begin{lemma}[Lov\'asz Local Lemma~\cite{EL75}]\label{local lemma}
Let $p\in[0,1)$ and $\mathcal A$ a finite set of events such that for every $A\in\mathcal A$,
\begin{itemize}
\item $\Prob{A} \leq p$, and
\item $A$ is mutually independent of a set of all but at most $d$ other events in $\mathcal A$.
\end{itemize}
If $4pd\leq 1$, then the probability that none of the events in $\mathcal A$ occur is strictly positive.
\end{lemma}

Next we need a concentration inequality of Bruhn and Joos~\cite{BrJo18}, derived from Talagrand's inequality~\cite{T95}.  To that end, we introduce the following definition.

\begin{definition}
  Let $((\sampleSpace_i, \sigmaAlg_i, \mathbb P_i))$ be probability spaces, and let $(\sampleSpace, \sigmaAlg, \mathbb P)$ be their product space.  
  We say a random variable $\X : \sampleSpace \rightarrow \mathbb R$ has \textit{upward $(s, \change)$-certificates} with respect to a set of exceptional outcomes $\sampleSpace^* \subseteq \sampleSpace$ if for every $\omega\in\sampleSpace\setminus\sampleSpace^*$ and every $t > 0$, there exists an index set $I$ of size at most $s$ so that $\X(\omega') \geq \X(\omega) - t$ for every $\omega' \in \sampleSpace\setminus \sampleSpace^*$ for which the restrictions $\omega|_I$ and $\omega'|_I$ differ in at most $t/\change$ coordinates.
\end{definition}

\begin{theorem}[Bruhn and Joos~\cite{BrJo18}]\label{thm:bj-tala}
  Let $((\sampleSpace_i, \sigmaAlg_i, \mathbb P_i))$ be probability spaces, let $(\sampleSpace, \sigmaAlg, \mathbb P)$ be their product space, and let $\sampleSpace^*\in \sigmaAlg$ be a set of exceptional outcomes.  Let $\X : \sampleSpace \rightarrow \mathbb R$ be a non-negative random variable, let $M \coloneqq \max\{\sup \X, 1\}$, and let $\change \geq 1$.  If $\Prob{\sampleSpace^*} \leq M^{-2}$ and $\X$ has upward $(s, \change)$-certificates, then for $t > 50 \change\sqrt s$,
  \begin{equation*}
    \Prob{|\X - \Expect{\X}| \geq t} \leq 4\exp\left(-\frac{t^2}{16\change^2s}\right) + 4\Prob{\sampleSpace^*}.
  \end{equation*}
\end{theorem}

\section{Proof of Lemma~\ref{lemma:nibble}}\label{section:nibble}

This section is devoted to the proof of Lemma~\ref{lemma:nibble}.  It will be convenient to assume that equality holds in \ref{hypo:containment-bound} and \ref{hypo:degree-bound} of Theorem~\ref{thm:main} and that moreover $d_{G_i, L}(v, c) = D$ for every $v \in V(G)$ and $c \in L(v)$, so we first prove the following proposition which enables us to consider an embedding of $G$ for which these properties hold.  

\begin{restatable}{proposition}{EmbeddingProp}\label{prop:embedding}
  Let $C, \Lambda, D \in \mathbb N$, let $G_1, \dots, G_m$ be nearly disjoint graphs, and let $L$ be a list assignment for $G \coloneqq \bigcup_{i=1}^m G_i$.  If $|\{i \in [m] : v \in V(G_i)\}| \leq C$ and $|L(v)| = \Lambda$ for every $v \in V(G)$ and $\Delta(G_i, L|_{V(G_i)}) \leq D$ for every $i \in [m]$, then there exist nearly disjoint graphs $G'_1, \dots, G'_{m'}$ for some $m' \geq m$ and a list assignment $L'$ for $G' \coloneqq \bigcup_{i=1}^{m'}G'_i$ such that 
  \begin{enumerate}[(\ref*{prop:embedding}.1)]
  \item\label{embedding:embedding} $G_i \subseteq G'_i$ for every $i \in [m]$,
  \item\label{embedding:lists-preserved} $L'(v) = L(v)$ for every $v \in V(G)$,
  \item\label{embedding:containment} $|\{i \in [m'] : v \in V(G'_i)\}| = C$ for every $v \in V(G')$,
  \item\label{embedding:list-size} $|L'(v)| = \Lambda$ for every $v \in V(G')$, and
  \item\label{embedding:degree} $d_{G'_i, L'}(v, c) = D$ for every $i \in [m']$ and $v \in V(G'_i)$ with $c \in L'(v)$.
  \end{enumerate}
\end{restatable}

The proof of Proposition~\ref{prop:embedding} can be found in the appendix\NOTAPPENDIX{ of the arxiv version of this paper}.
For the remainder of this section, we let $C,\eps > 0$, and we let $D$ be sufficiently large, we let $p$ satisfy $\log^{-1} D \geq p \geq \log^{-2}D$, and we let $G\coloneqq \bigcup_{i=1}^mG_i$ with list assignment $L$ satisfying the hypotheses of Lemma~\ref{lemma:nibble}.  For convenience, we assume $\Lambda$ and $D$ are integers where it does not affect the argument.  By Proposition~\ref{prop:embedding}, we may assume without loss of generality that $|\{i \in [m] : v \in V(G_i)\}| = C$ for every $v\in V(G)$ and that $d_{G_i, L}(v, c) = D$ for every $i \in [m]$ and $v \in V(G)$ with $c \in L(v)$\COMMENT{by applying the lemma to $G'$ and $L'$ from Proposition~\ref{prop:embedding} to obtain $X$, $\phi$, and $L''$ (say), and considering $X \cap V(G)$, $\phi|_{V(G)}$, and $L''_{V(G)}$ instead}.

For every pair $(A, \psi)$ satisfying $A \subseteq V(G)$ and $\psi(v) \in L(v)$ for each $v\in V(G)$,
\begin{itemize}
\item let $L_{A, \psi}(v) \coloneqq L(v) \setminus \{\psi(u) : u \in N_G(v)\cap A\}$ for every $v \in V(G)$, and
\item let $X_{A, \psi} \coloneqq \{v \in A : \psi(v) \in L_{A, \psi}(v)\}$.
\end{itemize}
Note that $X_{A, \psi} = A \setminus \bigcup_{v \in A}\{u \in N_G(v) : \psi(u) = \psi(v)\}$.

We will consider the probability space $(\sampleSpace, \sigmaAlg, \mathbb P)$ of such pairs $(A, \psi)$ where
\begin{itemize}
\item each vertex in $G$ is in $A$ independently and uniformly at random with probability $p$ and
\item $\psi(v) \in L(v)$ is chosen independently and uniformly at random for each $v \in V(G)$.
\end{itemize}
Note that $(\sampleSpace, \sigmaAlg, \mathbb P)$ is the product space of $(\sampleSpace^{\mathrm{act}}_{v}, \sigmaAlg^{\mathrm{act}}_{v}, \mathbb P^{\mathrm{act}}_{v})$ and $(\sampleSpace^{\mathrm{col}}_{v}, \sigmaAlg^{\mathrm{col}}_{v}, \mathbb P^{\mathrm{col}}_{v})$ taken over all $v \in V(G)$, where $(\sampleSpace^{\mathrm{act}}_{v}, \sigmaAlg^{\mathrm{act}}_{v}, \mathbb P^{\mathrm{act}}_{v})$ models whether $v \in A$ and $(\sampleSpace^{\mathrm{col}}_{v}, \sigmaAlg^{\mathrm{col}}_{v}, \mathbb P^{\mathrm{col}}_{v})$ models the choice of $\psi(v)$.  We will use this fact to apply Theorem~\ref{thm:bj-tala}.

We will use the Lov\' asz Local Lemma to show that with nonzero probability, $X = X_{A, \psi}$ and $L' = L_{A, \psi}|_{V(G)\setminus X}$ (with lists possibly truncated) satisfy the lemma with $\phi = \psi|_X$.
To that end, we define the following random variables for each vertex $v\in V(G)$ and $c\in L(v)$ and $i \in [m]$:
\begin{itemize}
\item $\remainingColoursRV_v(A, \psi) \coloneqq \left|L_{A,\psi}(v)\right|$,
\item $\colDegRV_{v, c, i}(A, \psi) \coloneqq \left|\left\{u \in N_{G_i}(v) : c \in L_{A, \psi}(u)\right\} \setminus X_{A, \psi}\right|$,
\item $\uncolActivatedNbrsRV_{v, c, i}(A, \psi) \coloneqq \left|\left\{u \in N_{G_i}(v) \cap A : c \in L(u)\right\} \setminus X_{A, \psi}\right|$,
\item $\keptUnactColsRV_{v, c, i}(A, \psi) \coloneqq \left|\left\{u \in N_{G_i}(v) \setminus A : c \in L(u)\right\} \setminus \bigcup_{u \in A\setminus V(G_i) : \psi(u) = c}N_G(u)\right|$.
\end{itemize}
Note that\COMMENT{The terms in the first line together count the number of $u \in N_{G_i}(v) \cap A$ with $c \in L(u)$ that are not in $X_{A, \psi}$ and have $c \in  L_{A, \psi}(u)$.  The third term counts the number of $u \in N_{G_i}(v)\setminus A$ with $c \in L_{A, \psi}(u)$, and this quantity is at most $\keptUnactColsRV_{v, c, i}(A, \psi)$.}
\begin{multline}
  \label{eqn:degree-overcount}
  \colDegRV_{v, c, i}(A, \psi) = \uncolActivatedNbrsRV_{v, c, i}(A, \psi) - |\{u \in (N_{G_i}(v) \cap A)\setminus X_{A, \psi} : c \in L(u)\} \setminus \bigcup_{u \in A : \psi(u) = c}N_G(u)\}|  \\
  + |\{u \in N_{G_i}(v) \setminus A : c \in L(u)\} \setminus \bigcup_{u \in A : \psi(u) = c}N_G(u)\}|
  \leq \uncolActivatedNbrsRV_{v, c, i}(A, \psi) + \keptUnactColsRV_{v, c, i}(A, \psi).
\end{multline}

To apply the Lov\'asz Local Lemma, we first show that for each $v \in V(G)$, we have that $\remainingColoursRV_v$ is at least the quantity in \ref{nibble:list-size} with high probability, and for each $c \in L(v)$ and $i \in [m]$ such that $v \in V(G_i)$, we have that $\colDegRV_{v, c, i}$ is at most the quantity in \ref{nibble:color-degree} with high probability.  To that end, we first compute the expectation of these random variables and then apply Theorem~\ref{thm:bj-tala}.  However, it is not possible to apply Theorem~\ref{thm:bj-tala} to $\colDegRV_{v, c, i}$ directly to obtain any meaningful bound.  (Consider the case $C = m = 1$ when $G_1$ is complete.  Either no vertex in $A$ is assigned $c$, in which case $c \in L_{A, \psi}(u)$ for every vertex $u$, or some vertex  in $A$ is assigned $c$, in which case $c \notin L_{A, \psi}(u)$ for all $u \in V(G)$.)  Nevertheless, we can apply Theorem~\ref{thm:bj-tala} to both $\uncolActivatedNbrsRV_{v, c, i}$ and $\keptUnactColsRV_{v, c, i}$ and moreover show that $\uncolActivatedNbrsRV_{v, c, i} + \keptUnactColsRV_{v, c, i}$ is at most the quantity in \ref{nibble:color-degree} with high probability.  Hence, \eqref{eqn:degree-overcount} implies the required bound for $\colDegRV_{v, c, i}$ as well.

In order to compute the expected values of $\remainingColoursRV_v$, $\uncolActivatedNbrsRV_{v, c, i}$, and $\keptUnactColsRV_{v, c, i}$, we need the following simple proposition.
\begin{proposition}\label{prop:keep}
  For every $i \in [m]$, $v\in V(G_i)$, and $c \in L(v)$,
  \begin{equation}
    \label{eqn:keep-probability}
    \Prob{\not\exists u \in N_{G_i}(v) \cap A \text{ such that }\psi(u) = c} = \left(1 - \frac{p}{\Lambda}\right)^D.
  \end{equation}
  Moreover, for every $v\in V(G)$ and $c\in L(v)$,
  \begin{equation}\label{eqn:full-keep-probability}
    \Prob{c \in L_{A, \psi}(v)} = \left(1 - \frac{p}{\Lambda}\right)^{DC}.
  \end{equation}
\end{proposition}
\begin{proof}
  First we prove \eqref{eqn:keep-probability}.
  For each $u \in N_{G_i}(v)$ with $c \in L(u)$, we have $\Prob{u \in A} = p$ and $\Prob{\psi(u) = c} = 1 / \Lambda$, and the events that $u \in A$ and that $\psi(u) = c$ are independent.  Moreover, these events are mutually independent for all such $u$, so
  \begin{equation*}
    \Prob{\not\exists u \in N_{G_i}(v) \cap A \text{ s.t.~}\psi(u) = c} = \prod_{u \in N_{G_i}(v) : c \in L(u)}\left(1 - \Prob{u \in A}\cdot\Prob{\psi(u) = c}\right) = \left(1 - \frac{p}{\Lambda}\right)^D,
  \end{equation*}
  as desired.

  Now we prove \eqref{eqn:full-keep-probability}.
  For every $v\in V(G)$ and $c \in L(v)$, we have $c \in L_{A, \psi}(v)$ if and only if there is no $u \in N_{G_i}(v) \cap A$ with $\psi(u) = c$ for every $i$ for which $v \in V(G_i)$.  By \ref{hypo:intersection-bound}, these events are mutually independent for all such $i$ (of which there are $C$ by \ref{hypo:containment-bound}), so by \eqref{eqn:keep-probability},
  \begin{equation*}
    \Prob{c\in L_{A, \psi}(v)} = \prod_{i \in [m] : v \in V(G_i)}\Prob{\not\exists u \in N_{G_i}(v) \cap A \text{ s.t.~}\psi(u) = c} = \left(1 - \frac{p}{\Lambda}\right)^{DC},
  \end{equation*}
  as desired.
\end{proof}
Now we compute the expectations of our random variables.
\begin{lemma}\label{lemma:expectations}
  Every vertex $v\in V(G)$ satisfies
  \begin{equation}\label{eqn:expected-list-size}
    \Expect{\remainingColoursRV_v} = \left(1 - \frac{p}{\Lambda}\right)^{DC}\Lambda = \newlambda(\Lambda, D, C, p).
  \end{equation}
  Moreover, for every $i \in [m]$, $v \in V(G_i)$, and $c \in L(v)$,
  \begin{align}
    \label{eqn:expected-uncol-active}
    &&\Expect{\uncolActivatedNbrsRV_{v, c, i}} &= p\left(1 - \left(1 - \frac{p}{\Lambda}\right)^{DC}\right)D && \text{and}\\
    \label{eqn:expected-removed-colors}
    &&\Expect{\keptUnactColsRV_{v, c, i}} &= (1 - p)\left(1 - \frac{p}{\Lambda}\right)^{D(C - 1)}D.
  \end{align}
\end{lemma}
\begin{proof}
  By the linearity of expectation, we have $\Expect{\remainingColoursRV_v} = \sum_{c\in L(v)}\Prob{c \in L_{A, \psi}(v)}$, so \eqref{eqn:expected-list-size} follows from \eqref{eqn:full-keep-probability}.

  Again by the linearity of expectation, we have
  \begin{equation*}
    \Expect{\uncolActivatedNbrsRV_{v, c, i}} = \sum_{u \in N_{G_i}(v) : c \in L(u)}\Prob{u \in A \text{ and } u \notin X_{A, \psi}}.
  \end{equation*}
  Note that $\Prob{u \in A \text{ and } u \notin X_{A, \psi}} = \Prob{u \in A \text{ and } \psi(u) \notin L_{A, \psi}(u)}$, and the events that $u \in A$ and that $\psi(u) \notin L_{A, \psi}(u)$ are independent.  We have $\Prob{u \in A} = p$ and $\Prob{\psi(u) \notin L_{A, \psi}(u)} = 1 - \left(1 - p / \Lambda\right)^{DC}$ by \eqref{eqn:full-keep-probability}, so \eqref{eqn:expected-uncol-active} follows from the equation above.
  \COMMENT{
    \begin{equation*}
      \Expect{\uncolActivatedNbrsRV_{v, c, i}} = \sum_{u \in N_{G_i}(v) : c \in L(u)}p\left(1 - \left(1 - \frac{p}{\Lambda}\right)^{DC}\right) = p\left(1 - \left(1 - \frac{p}{\Lambda}\right)^{DC}\right)D.
    \end{equation*}}

  By the linearity of expectation, we have
  \begin{equation*}
    \Expect{\keptUnactColsRV_{v, c, i}} = \sum_{u \in N_{G_i}(v) : c \in L(u)}\Prob{u \notin A \text{ and } \psi(w) \neq c~\forall w \in N_{G_j}(u) \cap A,~j\in[m]\setminus\{i\}}.
  \end{equation*}
  By \eqref{eqn:keep-probability}, $\Prob{\psi(w) \neq c~\forall w \in N_{G_j}(u) \cap A,~j\in[m]\setminus\{i\}} = \left(1 - p / \Lambda\right)^{D(C - 1)}$.  Since the events that $u \notin A$ and $\psi(w) \neq c~\forall w \in N_{G_j}(u)\cap A,~j\in[m]\setminus\{i\}$ are independent, \eqref{eqn:expected-removed-colors} follows from the equation above.
  \COMMENT{
    \begin{equation*}
      \Expect{\keptUnactColsRV_{v, c, i}} = \sum_{u \in N_{G_i}(v) : c \in L(u)}(1 - p)\left(1 - \frac{p}{\Lambda}\right)^{D(C - 1)} = (1 - p)\left(1 - \frac{p}{\Lambda}\right)^{D(C - 1)}D.
  \end{equation*}}
\end{proof}

Now we need to show that the random variables in Lemma~\ref{lemma:expectations} are close to their expectation with high probability.  We use Theorem~\ref{thm:bj-tala}, the `exceptional outcomes' version of Talagrand's Inequality.  To that end, we define an exceptional outcome for each vertex and show that it is unlikely.
First, for each vertex $v\in V(G)$ and $c\in L(v)$ and $i \in [m]$, we define the random variable
\begin{itemize}
\item $\conflictsRV_{v, c, i}(A, \psi) \coloneqq \left|\left\{u \in N_{G_i}(v) : \psi(u) = c\right\}\right|$.\COMMENT{We really only care about the vertices in $A$, but it is conceptually simpler to ignore that, as it does not greatly affect the probability of the exceptional outcomes.}
\end{itemize}
Then, for each vertex $v\in V(G)$ and $i \in [m]$, we define
\begin{itemize}
\item $\sampleSpace^*_{v, i} = \{(A, \psi) : \exists u \in \{v\} \cup N_{G_i}(v) \cup N^2_{G_i}(v),~\exists c\in L(u),~\conflictsRV_{u, c, i}(A, \psi) \geq \log D \}$,
\end{itemize}
where $N^2_{G_i}(v)$ is the set of vertices at distance two from $v$ in $G_i$.
To apply Theorem~\ref{thm:bj-tala} to $\uncolActivatedNbrsRV_{v, c, i}$, we will show that $\uncolActivatedNbrsRV_{v, c, i}$ has upward $(s, \delta)$-certificates with respect to exceptional outcomes $\sampleSpace^*_{v, i}$ with $s = 4D$ and $\delta = \log D$.  It turns out that we need to consider these exceptional outcomes because, for example, if there is a set of more than $\log D$ vertices in $N_{G_i}(v) \cap A$ that are all assigned the same color, then changing the outcome of a single trial (in particular the trial determining either $\psi(u)$ or whether $u \in A$ for some $u \in \{v\} \cup N_{G_i}(v) \cup N^2_{G_i}(v)$ for which $\conflictsRV_{u, c, i}(A, \psi) \geq \log D$ for some $c\in L(u)$) can affect the value of $\uncolActivatedNbrsRV_{v, c, i}$ by more than $\log D$ (see the final part of the proof of Lemma~\ref{lemma:concentrations} for more details).  We remark that the event $\sampleSpace^*_{v, i}$ contains more outcomes than is necessary; it is simply more convenient to consider $\sampleSpace^*_{v, i}$, and the argument adapts more easily for correspondence coloring in Section~\ref{section:corr-col}.

Now we bound the probability of these exceptional events.  For this, we assume without loss of generality that $uv \in E(G)$ only if $L(u) \cap L(v) \neq \emptyset$, so each $G_i$ has maximum degree at most $\Lambda D$.

\begin{proposition}\label{prop:exceptional-outcomes}
  Every vertex $v\in V(G)$ and $i \in [m]$ satisfies
  \begin{equation*}
    \Prob{\sampleSpace^*_{v, i}} \leq 11^3 C^3D^5\left(\log D\right)^{- \log D}.
  \end{equation*}
\end{proposition}
\begin{proof}
  First we bound the probability that $\conflictsRV_{u, c, i}$ is too large for each $u \in V(G)$ and $c\in L(u)$, as follows:
  \begin{equation*}
    \Prob{\conflictsRV_{u, c, i} \geq
      \log D} \leq \sum_{i = \lceil \log D\rceil}^{D}\binom{D}{i}\left(\frac{1}{\Lambda}\right)^i \leq
    \sum_{i=\lceil \log D\rceil}^{D}\left(\frac{eD}{i\Lambda }\right)^i \leq
    \sum_{i=\lceil \log D\rceil}^{D}\left(\frac{e}{i}\right)^i.
  \end{equation*}
  Since each term in the sum is at most $(e/\log D)^{\log D}$ and there are at most $D$ terms, it follows that
  \begin{equation*}
    \Prob{\conflictsRV_{u, c, i} \geq \log D} \leq D^2\left(\log D\right)^{- \log D}.
  \end{equation*}
  Thus, combining the above inequality with the Union Bound, we have $\Prob{\sampleSpace^*_{v,i}} \leq \Lambda(1 + \Lambda D + \Lambda^2D^2)\left(\log D\right)^{- \log D} \leq 11^3 C^3D^5\left(\log D\right)^{- \log D}$,
  as desired.
\end{proof}

Combining Theorem~\ref{thm:bj-tala} and Proposition~\ref{prop:exceptional-outcomes}, we show that $\remainingColoursRV_v$, $\uncolActivatedNbrsRV_{v, c, i}$, and $\keptUnactColsRV_{v, c, i}$ are close to their expectation with high probability in the following lemma.

\begin{lemma}\label{lemma:concentrations}
  Every vertex $v\in V(G)$ satisfies
  \begin{equation}
    \label{eqn:list-size-concentration}
    \Prob{\left|\remainingColoursRV_v - \Expect{\remainingColoursRV_v}\right| > D^{2/3}} \leq \exp\left(-D^{1/4}\right).
  \end{equation}
  Moreover, if $v \in V(G_i)$ and $c \in L(v)$, then
  \begin{align}
    \label{eqn:uncol-active-concentration}
    &&\Prob{\left|\uncolActivatedNbrsRV_{v, c, i} - \Expect{\uncolActivatedNbrsRV_{v, c, i}}\right| > D^{2/3}} &\leq \exp\left(-D^{1/4}\right)&&\text{and}\\
    \label{eqn:removed-colors-concentration}
    &&\Prob{\left|\keptUnactColsRV_{v, c, i} - \Expect{\keptUnactColsRV_{v, c, i}}\right| > D^{2/3}} &\leq \exp\left(-D^{1/4}\right).
  \end{align}
\end{lemma}
\begin{proof}
  First we prove \eqref{eqn:list-size-concentration}.  We apply Theorem~\ref{thm:bj-tala} with exceptional outcomes $\sampleSpace^* = \emptyset$.  To that end, we show that $\Lambda - \remainingColoursRV_v$ has upward $(s, \delta)$-certificates with respect to $\sampleSpace^*$, where $s = 2\Lambda$ and $\delta = 1$.  Let $(A, \psi) \in \sampleSpace$.  For every $c \in L(v) \setminus L_{A, \psi}(v)$, there is at least one neighbor $u \in N_{G}(v)$ of $v$ such that $u \in A$ and $\psi(u) = c$.  Choose one such neighbor, and denote it by $u_c$.  Let $I_{A, \psi}$ index the trials determining whether $u_c \in A$ and $\psi(u_c) = c$ for each $c \in L(v)\setminus L_{A, \psi}(v)$.  Now if $(A', \psi') \in \sampleSpace$ differs from $(A, \psi)$ in at most $t$ of the trials indexed by $I_{A, \psi}$ (and differs arbitrarily for trials not indexed by $I_{A, \psi}$), then all but at most $t$ colors in $L(v) \setminus L_{A, \psi}(v)$ are also in $L(v) \setminus L_{A', \psi'}(v)$.  Hence, $\Lambda - \remainingColoursRV_v(A', \psi') \geq \Lambda - \remainingColoursRV_v(A, \psi) - t$, and since $I_{A, \psi} \leq 2\Lambda = s$, it follows that $\Lambda - \remainingColoursRV_v$ has upward $(s, \delta)$-certificates, as desired.  Therefore by Theorem~\ref{thm:bj-tala} with $t = D^{2/3}$\COMMENT{since $\Lambda \leq 10CD$, we have $t > 50\delta\sqrt s$, as required.},
  \begin{equation*}
    \Prob{\left|\remainingColoursRV_v - \Expect{\remainingColoursRV_v}\right| > D^{2/3}} \leq 4\exp\left(-\frac{D^{4/3}}{32\Lambda}\right) \leq\exp\left(-D^{1/4}\right),
  \end{equation*}
  as desired.

  Now we prove \eqref{eqn:removed-colors-concentration}.  We apply Theorem~\ref{thm:bj-tala} with exceptional outcomes $\sampleSpace^* = \emptyset$.  To that end, we show that $D - \keptUnactColsRV_{v, c, i}$ has upward $(s, \delta)$-certificates with respect to $\sampleSpace^*$, where $s = 3D$ and $\delta = 1$.  Let $(A, \psi) \in \sampleSpace$, and note that
  \begin{equation*}
    D - \keptUnactColsRV_{v, c, i}(A, \psi) = |\{x \in N_{G_i}(v) : c \in L(x)\} \cap (A \cup \bigcup_{w \in A\setminus V(G_i) : \psi(w) = c}N_G(w))|.
  \end{equation*}
  For every $u \in \{x \in N_{G_i}(v) : c \in L(x)\} \cap \bigcup_{w \in A\setminus V(G_i) : \psi(w) = c}N_G(w)$, there is at least one neighbor $w \in N_{G_j}(u)$ of $u$ where $j \in [m]\setminus\{i\}$ such that $w \in A$ and $\psi(w) = c$.  Choose one such neighbor, and denote it by $w_u$.  By \ref{hypo:intersection-bound}, the vertices $w_u$ are distinct for different choices of $u$.  Let $I_{A, \psi}$ index the trials determining whether $u \in A$ for each $u \in \{x \in N_{G_i}(v) : c \in L(x)\}$ and whether $w_u \in A$ and $\psi(w_u) = c$ for each $u \in \{x \in N_{G_i}(v) : c \in L(x)\} \cap \bigcup_{w \in A\setminus V(G_i) : \psi(w) = c}N_G(w)$.  Now if $(A', \psi') \in \sampleSpace$ differs from $(A, \psi)$ in at most $t$ of the trials indexed by $I_{A, \psi}$ (and differs arbitrarily for trials not indexed by $I_{A, \psi}$)\COMMENT{since the $w_u$ vertices are distinct}, all but at most $t$ vertices in $\{x \in N_{G_i}(v) : c \in L(x)\} \cap (A \cup \bigcup_{w \in A\setminus V(G_i) : \psi(w) = c}N_G(w))$ are in $\{x \in N_{G_i}(v) : c \in L(x)\} \cap (A' \cup \bigcup_{w \in A'\setminus V(G_i) : \psi'(w) = c}N_G(w))$.  Hence, $D - \keptUnactColsRV_{v, c, i}(A', \psi') \geq D - \keptUnactColsRV_{v, c, i}(A, \psi) - t$, and since $I_{A, \psi} \leq 3D = s$, it follows that $\keptUnactColsRV_{v, c, i}$ has upward $(s, \delta)$-certificates, as desired.  Therefore by Theorem~\ref{thm:bj-tala} with $t = D^{2/3}$,
  \begin{equation*}
    \Prob{\left|\keptUnactColsRV_{v, c, i} - \Expect{\keptUnactColsRV_{v, c, i}}\right| > D^{2/3}} \leq 4\exp\left(-\frac{D^{4/3}}{48D}\right) \leq\exp\left(-D^{1/4}\right),
  \end{equation*}
  as desired.
  
  Finally, we prove \eqref{eqn:uncol-active-concentration}. 
We apply Theorem~\ref{thm:bj-tala} with exceptional outcomes $\sampleSpace^* = \sampleSpace^*_{v, i}$.
To that end, we show that $\uncolActivatedNbrsRV_{v, c, i}$ has upward $(s, \delta)$-certificates with respect to $\sampleSpace^*$, where $s = 4D$ and $\delta = \log D$.  Let $(A, \psi) \in \sampleSpace\setminus\sampleSpace^*$.  For every $u \in \{x \in N_{G_i}(v) \cap A : c \in L(x)\} \setminus X_{A, \psi}$, there is at least one neighbor $w \in N_{G}(u)$ of $u$ such that $w \in A$ and $\psi(w) = \psi(u)$.  Choose one such neighbor, and denote it by $w_u$.  Let $I_{A, \psi}$ index the trials determining whether $u \in A$, whether $w_u \in A$, and the assignment of $\psi(u)$ and $\psi(w_u)$ for each $u \in \{x \in N_{G_i}(v) \cap A : c \in L(x)\}$.  Since $(A, \psi)\notin\sampleSpace^*_v$,
the multi-set $\{w_u : u \in (N_{G_i}(v) \cap A) \setminus X_{A, \psi},\ c \in L(u)\}$ has maximum multiplicity at most $\log D$ (as otherwise, if $w$ appears more than $\log D$ times in this multi-set, then $\conflictsRV_{w, \psi(w), i}(A, \psi) > \log D$).
Hence, if $(A', \psi') \in \sampleSpace$ differs from $(A, \psi)$ in at most $t / \log D$ of the trials indexed by $I_{A, \psi}$ (and differs arbitrarily for trials not indexed by $I_{A, \psi}$), then all but at most $t$ vertices in $\{x \in N_{G_i}(v) \cap A : c \in L(x)\} \setminus X_{A, \psi}$ are in $\{x \in N_{G_i}(v) \cap A' : c \in L(x)\} \setminus X_{A', \psi'}$.  Hence, $\uncolActivatedNbrsRV_{v, c, i}(A', \psi') \geq \uncolActivatedNbrsRV_{v, c, i}(A, \psi) - t$, and since $I_{A, \psi} \leq 4D = s$, it follows that $\uncolActivatedNbrsRV_{v, c, i}$ has upward $(s, \delta)$-certificates, as desired.  Therefore by Theorem~\ref{thm:bj-tala} with $t = D^{2/3}$ and Proposition~\ref{prop:exceptional-outcomes},
  \begin{equation*}
    \Prob{\left|\uncolActivatedNbrsRV_{v, c, i} - \Expect{\uncolActivatedNbrsRV_{v, c, i}}\right| > D^{2/3}} \leq 4\exp\left(-\frac{D^{4/3}}{64D\log^2 D}\right) + 44^3C^3D^5\left(\log D\right)^{-\log D}\leq\exp\left(-D^{1/4}\right),
  \end{equation*}
  as desired.
\end{proof}

Finally we can prove Lemma~\ref{lemma:nibble}, using Lemmas~\ref{lemma:expectations} and~\ref{lemma:concentrations}.
\begin{proof}[Proof of Lemma~\ref{lemma:nibble}]
  First, rather than showing \ref{nibble:list-size}, it suffices to show that $|L'(v)| \geq \left(1 - p/\Lambda\right)^{DC}\Lambda - \Lambda^{4/5}$ for every $v \in V(G)\setminus X$, since we can truncate each list until equality holds.
  
  We consider $(A, \psi)$ chosen randomly as described earlier, and we define the following set of bad events for each vertex $v\in V(G)$ and $c\in L(v)$ and $i \in [m]$:
  \begin{align*}
    \mathcal A_v &= 
                     \left\{
                   (A, \psi) \,:\, \remainingColoursRV_v(A, \psi) < \Expect{\remainingColoursRV_v} - \Lambda^{4/5}\right\},\\
        \mathcal A_{v, c, i} &=
                        \left\{
                                (A, \psi) \,:\, \uncolActivatedNbrsRV_{v, c, i} > \Expect{\uncolActivatedNbrsRV_{v, c, i}} + D^{4/5}/2 \right\}, \text{ and}\\
    \mathcal A'_{v, c, i} &= \left\{(A, \psi) \,:\, \keptUnactColsRV_{v,c,i}(A, \psi) > \Expect{\keptUnactColsRV_{v, c, i}} + D^{4/5}/2\right\}.
  \end{align*}
  Letting $\mathcal A$ be the union of all such bad events,
  note that each event in $\mathcal A$ is mutually independent of all but at most $(\lceil\Lambda\rceil CD)^4 \leq (11C^2D^2)^4$ other events in $\mathcal A$.  
  By Lemma~\ref{lemma:concentrations}, every event in $\mathcal A$ occurs with probability at most $\exp\left(-D^{1/4}\right)$.
  Therefore by the Lov\' asz Local Lemma, there exists $(A, \psi)\notin \mathcal A$.

  Now we show that $X \coloneqq X_{A, \psi}$ and $L' \coloneqq L_{A, \psi}|_{V(G)\setminus X}$ satisfy the lemma with $\phi \coloneqq \psi|_X$.  Indeed, $\phi$ is an $L|_X$-coloring of $G[X]$ and since $X \subseteq A$, we have $L'(v) \subseteq L(v)\setminus \{\phi(u) : u \in N_G(v) \cap X\}$, as required.  Moreover, for every $v\in V(G)$, since $(A, \psi) \notin \mathcal A_v$, by \eqref{eqn:expected-list-size}, we have $|L'(v)| \geq \left(1 - p/\Lambda\right)^{DC}\Lambda - \Lambda^{4/5}$, as desired, and if $v \in V(G_i)$ and $c\in L(v)$, then since $(A, \psi) \notin \mathcal A_{v, c, i} \cup \mathcal A'_{v, c, i}$, by \eqref{eqn:degree-overcount}, \eqref{eqn:expected-uncol-active}, and \eqref{eqn:expected-removed-colors}, we have \ref{nibble:color-degree}, as desired.
\end{proof}

\section{Correspondence coloring}\label{section:corr-col}

In this section we introduce \textit{correspondence coloring} and describe how to generalize Theorem~\ref{thm:main} to this setting.

\begin{definition}\label{correspondence coloring}
  Let $G$ be a graph with list assignment $L$.

  \begin{itemize}
  \item If $M$ is a map with domain $E(G)$ where for each $e=uv\in E(G)$, $M(e)$ is a matching of $\{u\}\times L(u)$ and $\{v\}\times L(v)$, we say $(L, M)$ is a \textit{correspondence assignment} for $G$.	
  \item An \textit{$(L, M)$-coloring} of $G$ is a map $\phi$ with domain $V(G)$ such that $\phi(u)\in L(u)$ for every $u\in V(G)$, and every $e=uv\in E(G)$ satisfies $(u, \phi(u))(v, \phi(v))\notin M(e)$.  If $G$ has an $(L, M)$-coloring, then we say $G$ is \textit{$(L, M)$-colorable}.
  \item 
    The \textit{correspondence chromatic number} of $G$, also called the \textit{$DP$-chromatic number}, denoted $\chi_{DP}(G)$, is the minimum $k$ such that $G$ is $(L, M)$-colorable for every correspondence assignment $(L, M)$ satisfying $|L(v)| \geq k$ for all $v\in V(G)$.
  \end{itemize}
\end{definition}
For convenience, if $uv\in E(G)$, $c_1\in L(u)$, $c_2\in L(v)$, and $(u, c_1)(v, c_2)\in M(uv)$, we will just write $c_1c_2\in M(uv)$.  For each $v \in V(G)$ and $c \in L(v)$, we will let $N_{G, (L, M)}(v, c) \coloneqq \{(u, c') : uv \in E(G),~cc' \in M(uv)\}$, and we omit the subscript in $N_{G, (L, M)}(v, c)$ when it is clear from the context.
Note that if for each $e=uv\in E(G)$ and $c\in L(u)\cap L(v)$, we have $cc\in M(uv)$, then an $(L, M)$-coloring is an $L$-coloring.  Hence, every graph $G$ satisfies $\chi_\ell(G)\leq\chi_{DP}(G)$.

For a graph $G$ with correspondence assignment $(L, M)$, we define the \textit{color degree} of each $v\in V(G)$ and $c \in L(v)$ to be $d_{G, (L, M)}(v, c) \coloneqq |N_{G, (L, M)}(v, c)|$ and the \textit{maximum color degree} to be $\Delta(G, (L, M)) \coloneqq \max_{v\in V(G)}\max_{c \in L(v)}d_{G, (L, M)}(v, c)$.  With only minor modifications to the argument used to prove Theorem~\ref{thm:main}, which we describe below, we can strengthen Theorem~\ref{thm:main} to the setting of correspondence coloring, as follows.
\begin{theorem}\label{thm:main-corr}
    For every $C, \eps > 0$, the following holds for all sufficiently large $D$.  Let $G_1, \dots, G_m$ be graphs that
  \begin{enumerate*}[(i)]
  \item\label{hypo:intersection-bound-corr} are nearly disjoint and 
  \item\label{hypo:containment-bound-corr} satisfy $|\{i \in [m] : v \in V(G_i)\}| \leq C$ for every $v \in \bigcup_{i=1}^m V(G_i)$.
  \end{enumerate*}
  If $(L, M)$ is a correspondence assignment for $G \coloneqq \bigcup_{i=1}^m G_i$ satisfying
  \begin{enumerate*}[(i)]\stepcounter{enumi}\stepcounter{enumi}
  \item\label{hypo:degree-bound-corr} $\Delta(G_i, (L|_{V(G_i)}, M|_{E(G_i)}) \leq D$ for every $i \in [m]$ and
  \item\label{hypo:list-bound-corr} $|L(v)| \geq (1 + \eps)D$ for every $v\in V(G)$,
  \end{enumerate*}
  then $G$ is $(L, M)$-colorable.
\end{theorem}

As mentioned, Theorem~\ref{thm:main-corr} implies that Corollary~\ref{cor:main} actually holds for the correspondence chromatic number, which in turn implies the main result of~\cite{M18}, that linear and uniform hypergraphs of maximum degree at most $D$ have \textit{correspondence chromatic index} at most $D + o(D)$.

To prove Theorem~\ref{thm:main-corr}, we use the argument presented in Section~\ref{section:main-proof}, but with Lemmas~\ref{lemma:nibble} and \ref{lemma:finishing-blow} replaced with the following lemmas, respectively.
\begin{lemma}
  \label{lemma:nibble-corr}
  For every $C, \eps > 0$, there exists $D_{\ref{lemma:nibble-corr}}$ such that the following holds for every $D \geq D_{\ref{lemma:nibble-corr}}$.
  Let $G_1, \dots, G_m$ be graphs, and let $(L, M)$ be a correspondence assignment for $G \coloneqq \bigcup_{i=1}^m G_i$ satisfying \ref{hypo:intersection-bound-corr}--\ref{hypo:list-bound-corr} (of Theorem~\ref{thm:main-corr}).
  If $|L(v)| = \lceil\Lambda\rceil$ for every $v\in V(G)$, where $(1 + \eps)D \le \Lambda \le 10CD$,  and if $\log^{-1} D \geq p \geq \log^{-2}D$,
  then there exist $X\subseteq V(G)$, an $(L|_X, M|_{E(G[X])})$-coloring $\phi$ of $G[X]$, and a correspondence assignment $L'$ for $G - X$ satisfying $L'(v) \subseteq \{c \in L(v): (v, c) \in (\{v\}\times L(v)) \setminus \bigcup_{u \in X}N(u, \phi(u))\}$ for every $v\in V(G)$,
  such that 
  \begin{enumerate}[(\ref*{lemma:nibble-corr}.1)]
  \item\label{nibble-corr:list-size} $|L'(v)| = \left\lceil \newlambda(\Lambda, D, C, p) - \Lambda^{4/5}\right\rceil$ for every $v\in V(G)\setminus X$ and
  \item\label{nibble-corr:color-degree} $\Delta(G_i - X, (L'|_{V(G_i - X)}, M|_{E(G_i - X)})) \leq \newdeg(\Lambda, D, C, p) + D^{4/5}$ for every $i \in [m]$.
  \end{enumerate}  
\end{lemma}

\begin{lemma}\label{lemma:finishing-blow-corr}
  Let $G$ be a graph with correspondence assignment $L$.  If $|L(v)| \geq 8D$ for every $v \in V(G)$ and $\Delta(G, (L, M)) \leq D$, then $G$ is $(L, M)$-colorable.
\end{lemma}

Lemma~\ref{lemma:finishing-blow-corr}, like Lemma~\ref{lemma:finishing-blow}, can be proved with a straightforward application of the Lov\' asz Local Lemma.  A stronger result (with `$|L(v)| \geq 8D$' replaced by `$|L(v)| \geq 2D$') also follows easily from a well-known result of Haxell~\cite{Hax01} on independent transversals.  

Therefore it remains to describe how the argument presented in Section~\ref{section:nibble} to prove Lemma~\ref{lemma:nibble} can be modified to obtain Lemma~\ref{lemma:nibble-corr}.  We consider the same probability space $(\sampleSpace, \sigmaAlg, \mathbb P)$ of pairs $(A, \psi)$, but we instead define $L_{A, \psi}(v) \coloneqq \{c \in L(v) : (v, c) \in \left(\{v\}\times L(v)\right) \setminus \bigcup_{u\in A}N(u, \psi(u))\}$.  The definition of $X_{A, \psi}$ and of $\remainingColoursRV_v(A, \psi)$ remains the same, but we replace the definitions of $\colDegRV_{v, c, i}(A, \psi)$, $\uncolActivatedNbrsRV_{v, c, i}(A, \psi)$, and $\keptUnactColsRV_{v, c, i}(A, \psi)$ with the following:
\begin{itemize}
\item $\colDegRV_{v, c, i}(A, \psi) \coloneqq \left|\left\{ (u, c') \in N_{G_i, (L|_{V(G_i)}, M|_{E(G_i)})}(v, c) : u \notin X_{A, \psi},~c' \in L_{A, \psi}(u)\right\}\right|$,
\item $\uncolActivatedNbrsRV_{v, c, i}(A, \psi) \coloneqq \left|\left\{ (u, c') \in N_{G_i, (L|_{V(G_i)}, M|_{E(G_i)})}(v, c) : u \in A \setminus X_{A, \psi}\right\}\right|$, and
\item $\keptUnactColsRV_{v, c, i}(A, \psi) \coloneqq \left|\left\{ (u, c') \in N_{G_i, (L|_{V(G_i)}, M|_{E(G_i)})}(v, c) : u \notin A\right\} \setminus \bigcup_{u \in A\setminus V(G_i)}N_{G, (L, M)}(u, \psi(u))\right|$.
\end{itemize}

The definition of $\sampleSpace^*_{v, i}$ remains the same, but we replace the definition of $\conflictsRV_{v, c, i}(A, \psi)$ with the following:
\begin{itemize}
\item $\conflictsRV_{v, c, i}(A, \psi) \coloneqq \left|\left\{(u, c') \in N_{G_i, (L_{V(G_i)}, M_{E(G_i)})}(v, c) : \psi(u) = c'\right\}\right|$. \COMMENT{We still do not need to consider whether or not vertices are activated here, but we could require $u \in A$.}
\end{itemize}

The remainder of the proof can be obtained via straightforward modifications, so we omit the details.\COMMENT{To modify Proposition~\ref{prop:embedding}, we construct a correspondence assignment $(L', M')$ for $G'$ satisfying \ref{embedding:embedding}--\ref{embedding:list-size} such that $d_{G'_i, (L', M')}(v, c) = D$ for every $i \in [m']$ and $v \in V(G'_i)$ with $c \in L'(v)$.  We may consider the same construction of $G'$ and $L'$ (in iteration $i$, say); however, it is simpler to instead define $L'((v, x_1, \dots, x_k)) \coloneqq L(v)$ for each $v \in V(G)$ and $(x_1, \dots, x_k) \in [2D]^k$.  For every $(x_1, \dots, x_k) \in [2D]^k$ and every $uv \in E(G)$, we let $cc' \in M'((u, x_1, \dots, x_k)(v, x_1, \dots, x_k))$ if and only if $cc' \in M(uv)$.  For every $v \in V(G_i)$ and every pair $(x_1, \dots, x_k), (y_1, \dots, y_k) \in [2D]^k$, we have $(v, x_1, \dots, x_k)(v, y_1, \dots, y_k) \in E(G')$ if and only if $x_\ell = y_\ell$ for every $\ell \in [k]\setminus\{j\}$ for some $j \in [k]$ such that $(v, x_1, \dots, x_j)(v, y_1, \dots, y_j) \in R_{j, v}$.  In this case, we define $M'((v, x_1, \dots, x_k)(v, y_1, \dots, y_k))$ to be the matching of size one with $c_jc_j \in M'((v, x_1, \dots, x_k)(v, y_1, \dots, y_k))$.  Now for every $(v, x_1, \dots, x_k) \in V(G_i) \times [2D]^k$ and $j \in [k]$ with $c_j \in L(v)$, we have
  \begin{equation*}
        d_{G'_i, (L', M')}\left((v, x_1, \dots, x_k), c_j\right) = 
        d_{G_i, (L, M)}(v, c_j)
        + d_{R_{j, v}}((v, x_1, \dots, x_j)) = D,      
  \end{equation*}
  and for every $(v, x_1, \dots, x_k) \in V(G_\ell) \times [2D]^k$ for $\ell \in [m]\setminus\{i\}$ and every $j \in [k]$ with $c_j \in L(v)$, we have
  \begin{equation*}
    d_{G'_\ell, (L', M')}\left((v, x_1, \dots, x_k), c_j\right) =
    d_{G, L}(v, c_j).
  \end{equation*}
  Thus, the result follows by induction.}
This completes the proof of Lemma~\ref{lemma:nibble-corr} and in turn implies Theorem~\ref{thm:main-corr}.

\section{Proof of Theorems~\ref{thm:chi-efl} and \ref{thm:postle}}\label{section:chi-efl}


In this section, we prove Theorems~\ref{thm:chi-efl} and \ref{thm:postle}.  Recall that $f(m, \cG)$ is the largest possible chromatic number of the union of at most $m$ nearly disjoint graphs in $\cG$, and $\cG^\chi_n$ is the set of graphs of chromatic number at most $n$.  We begin by providing the construction that certifies that $f(3, \cG^\chi_n) \geq n + 1$.

\begin{proof}[Proof of Theorem~\ref{thm:chi-efl}\ref{chi-efl-lower}]
    Let $H_1$ and $H_2$ be complete graphs on $n + 1$ vertices such that $V(H_1) \cap V(H_2) = \{v\}$ for some vertex $v$, let $u \in V(H_1) \setminus \{v\}$, let $w \in V(H_2) \setminus \{v\}$, let $G_1 \coloneqq H_1 - uv$, let $G_2 \coloneqq H_2 - wv$, and let $G_3$ be the graph consisting of the single edge $uw$.  It is straightforward to check that
  \begin{itemize}
  \item $G_1$, $G_2$, and $G_3$ are nearly disjoint,
  \item $\chi(G_i) \leq n$ for every $i \in [3]$, and
  \item $\chi(G_1 \cup G_2 \cup G_3) = n + 1$,
  \end{itemize}
  as desired.
\end{proof}

To prove Theorem~\ref{thm:chi-efl}, it remains to prove the bound $f(m, \cG^\chi_n) \leq m + n - 2$ when $m + n$ is sufficiently large.  First we note the following immediate consequence of the main result in \cite{KKKMO21} (namely that the Erd\H os--Faber--Lov\' asz conjecture holds for all sufficiently large $n$).

\begin{theorem}[\cite{KKKMO21}]\label{thm:efl}
  The following holds for all sufficiently large $C$: If $G_1, \dots, G_m$ are nearly disjoint graphs satisfying $\max\left\{m, |V(G_1)|, \dots, |V(G_m)|\right\} \leq C$, then $\chi\left(\bigcup_{i=1}^m G_i\right) \leq C$.
\end{theorem}

\begin{proof}[Proof of Theorem~\ref{thm:chi-efl}\ref{chi-efl-upper}]
  Suppose to the contrary, and let $G_1, \dots, G_m$ be nearly disjoint graphs, each of chromatic number at most $n$, such that $\chi(\bigcup_{i=1}^m G_i) > m + n - 2$ and $|V(\bigcup_{i=1}^m G_i)|$ is minimum.  Note that $n \geq 2$ as otherwise $\chi(\bigcup_{i=1}^m G_i) = 1$.  By Theorem~\ref{thm:efl} with $n + m - 2$ playing the role of $C$, there exists $i \in [m]$ such that $|V(G_i)| > m + n - 2$, and we may assume without loss of generality that $i = 1$.  Let $X \coloneqq V(G_1) \setminus \bigcup_{i=2}^m V(G_i)$.  Since $G_1, \dots, G_m$ are nearly disjoint,
  \begin{equation}\label{eqn:n-extra-vtcs}
    |X| \geq |V(G_1)| - (m - 1) \geq n.
  \end{equation}

  Since $\chi(G_1) \leq n$, there is a partition $I_1, \dots, I_n$ of $V(G_1)$ into independent sets.  By \eqref{eqn:n-extra-vtcs} and the pigeonhole principle, there exists some $j\in [n]$ such that either $|I_j \cap X| > 1$ or both $I_j \cap X$ and $I_j \setminus X$ are nonempty\COMMENT{If not the former, then equality holds in \eqref{eqn:n-extra-vtcs} and $|I_j \cap X| = 1$ for every $j \in [n]$, in which case any $j$ for which $I_j$ intersects $V(G_1)\setminus X$ suffices.}.  We may assume without loss of generality that $j = 1$, so there exist distinct vertices $u, v\in I_1$ such that $u \in X$.  Let $G'_1$ be the graph obtained from $G_1$ by identifying $u$ and $v$ into a single new vertex.  Crucially, $G'_1, G_2, \dots, G_m$ are nearly disjoint, $\chi(G'_1) \leq n$, and $\chi(G'_1 \cup \bigcup_{i=2}^m G_i) \geq \chi(\bigcup_{i=1}^m G_i)$, contradicting the choice of $G_1, \dots, G_m$ to have $|V(\bigcup_{i=1}^m G_i)|$ minimum.
\end{proof}

Now we prove Theorem~\ref{thm:postle}.  In this proof, we use a Latin square to construct a graph of large chromatic number.  An \textit{order-$n$ Latin square} is an $n\times n$ array of $n$ symbols where each row and column contains each symbol exactly once.  That is, for an order-$n$ Latin square $L$, with rows, columns, and symbols indexed by $\row$, $\col$, and $\sym$, respectively, we have $\{L(a, b) : b \in \col\} = \sym$ for every $a \in \row$ and $\{L(a, b) : a \in \row\} = \sym$ for every $b \in \col$.

\begin{proof}[Proof of Theorem~\ref{thm:postle}]
  First, let $G \cong K_{m}$, and since $3 \mid m$, there exist sets $\row, \col, \sym \subseteq V(G)$ of size $m / 3$ that partition $V(G)$.  Since $n \geq m - 1$, we can let $G'$ be the multigraph obtained from $G$ by adding $n - (m - 1) + m/ 3$ loops to each vertex in $\row \cup \col$ and $n - (m - 1)$ loops to each vertex in $\sym$, and we let $H$ be the line graph of $G'$.  Let $L$ be an order-$(m / 3)$ Latin square with rows, columns, and symbols indexed by $\row$, $\col$, and $\sym$, respectively, and for each $a \in \row$ and $b \in \col$, let $e^{ab}_1$ denote the edge in $H$ with ends $ab$ and $ac$, and let $e^{ab}_2$ denote the edge in $H$ with ends $ab$ and $bc$, where $c = L(a, b)$.  Let
  \begin{equation*}
    H' \coloneqq H - \bigcup_{(a,b) \in \row\times\col}\{e^{ab}_1, e^{ab}_2\}.
  \end{equation*}
  We prove that $H'$ has chromatic number at least $n + m / 6$ and is a nearly disjoint union of $m$ graphs of chromatic number $n$, as desired.

  To that end, for each $v \in V(G)$, let $G_v\coloneqq H'[\{e \in E(G') : e \ni v\}]$.  By construction, the graphs in $\{G_v : v \in V(G)\}$ are nearly disjoint, and $H' = \bigcup_{v \in V(G)} G_v$.  Moreover, $G_v$ is isomorphic to $K_{n + m/3}$ with a matching of size $m/3$ removed for every $v \in \row \cup \col$ and is isomorphic to $K_{n}$ for every $v \in \sym$, so $\chi(G_v) = n$ for every $v \in V(G)$, as required.

  To prove that $\chi(H') \geq n + m / 6$, we let $\phi$ be a proper coloring of $H'$ using at most $k$ colors, and we show that $k \geq n + m/6$, as follows.  First, let
  \begin{align*}
    \col_a &\coloneqq \{b \in \col : \phi(ab) = \phi(ac),\ \text{where } c = L(a, b)\} \text{ for every $a \in \row$} && \mathrm{and}\\
    \row_b &\coloneqq \{a \in \row : \phi(ab) = \phi(bc),\ \text{where } c = L(a, b)\} \text{ for every $b \in \col$}.
  \end{align*}
  We claim that the following holds:
  \begin{enumerate}[(a)]
  \item\label{item:mono-pairs-bound} $|\col_a| \geq n + m / 3 - k$ for every $a \in \row$,
  \item\label{item:mono-pairs-bound2} $|\row_b| \geq n + m / 3 - k$ for every $b \in \col$, and
  \item\label{item:mono-trips-bound} $\sum_{a\in \row}|\col_a| + \sum_{b \in \col}|\row_b| \leq \frac{m^2}{9}$.
  \end{enumerate}
  Altogether, \ref{item:mono-pairs-bound}--\ref{item:mono-trips-bound} imply that $(2m / 3)(n + m / 3 - k) \leq m^2 / 9$.  Rearranging terms in this inequality, we have $k \geq n + m / 6$, as desired.  Thus, it remains to prove \ref{item:mono-pairs-bound}--\ref{item:mono-trips-bound}.

  To prove \ref{item:mono-pairs-bound}, fix some $a \in \row$.  Since $\phi$ is a proper coloring, the colors assigned to $\{ab : b \in \col_a\}$ and the colors assigned to vertices of $G_a$ not incident to an edge of $\{e^{ab}_1 : b \in \col_a\}$ are distinct.  Thus, since at most $k$ colors are used in total, $|\col_a| + n + m / 3 - 2|\col_a| \leq k$.  Rearranging terms in this inequality, we have $|\col_a| \geq n + m / 3 - k$, as desired.  The proof of \ref{item:mono-pairs-bound2} is essentially the same, so we omit the details.

  To prove \ref{item:mono-trips-bound},  suppose to the contrary that $\sum_{a\in A}|\col_a| + \sum_{b\in B}|\row_b| > m^2 / 9$.  In this case, there is a pair $(a, b) \in \row \times \col$ such that $a \in \row_b$ and $b \in \col_a$.  Letting $c = L(a, b)$, since $a \in \row_b$, we have $\phi(ab) = \phi(bc)$, and since $b \in \row_a$, we have $\phi(ab) = \phi(ac)$.  However, since $\phi$ is a proper coloring, $\phi(bc) \neq \phi(ac)$, a contradiction.  
\end{proof}

\section*{Acknowledgements}

We thank Luke Postle for allowing us to include the proof of Theorem~\ref{thm:postle}.


\providecommand{\bysame}{\leavevmode\hbox to3em{\hrulefill}\thinspace}
\providecommand{\MR}{\relax\ifhmode\unskip\space\fi MR }
\providecommand{\MRhref}[2]{%
  \href{http://www.ams.org/mathscinet-getitem?mr=#1}{#2}
}
\providecommand{\href}[2]{#2}

\APPENDIX{
  \appendix
  \renewcommand\sectionname{}
  \section{Proof of Proposition~\ref{prop:embedding}}

  This section is devoted to the proof of Proposition~\ref{prop:embedding}, which we restate for the reader's convenience.
  \EmbeddingProp*
  \begin{proof}
    By possibly adding graphs $G_{m+1}, \dots $ to the list $G_1, \dots, G_m$, each consisting of a single vertex, we assume without loss of generality that $|\{i \in [m] : v \in V(G_i)\}| = C$ for every $v \in V(G)$ (that is, that \ref{embedding:containment} already holds for $m' \coloneqq m$, $G'_i \coloneqq G_i$ for $i \in [m]$, and $G' \coloneqq G$).  

    First enumerate the colors in $\bigcup_{v \in V(G)}L(v)$ as $c_1, \dots, c_k$.
    For each $i \in [m]$, we iteratively do the following.  We first construct nearly disjoint $G'_1, \dots, G'_m$ satisfying \ref{embedding:embedding} and \ref{embedding:containment}.  We also construct a list assignment $L'$ for $G' \coloneqq \bigcup_{\ell=1}^mG'_\ell$ satisfying \ref{embedding:lists-preserved} and \ref{embedding:list-size} such that $d_{G'_\ell, L'}(v, c) = D$ for every $\ell \in [i]$ and $v \in V(G'_\ell)$ with $c \in L'(v)$ (i.e.\ \ref{embedding:degree} holds for $G'_1, \dots, G'_i$).
    We apply each subsequent iteration to embed the graphs $G'_1, \dots, G'_m$ (i.e.\ we relabel $G'_\ell$ as $G_\ell$ for the next iteration and repeat with $i + 1$), and at the end of the procedure, \ref{embedding:embedding}--\ref{embedding:degree} all hold, as desired.

    In the $i$th iteration, for each $\ell \in [m]$, inductively construct $G'_{\ell, 0}, \dots, G'_{\ell, k}$ as follows.  Let $G'_{\ell, 0} \coloneqq G_\ell$, and for each $j \in [k]$, let $G_{\ell, j}$ consist of $2D$ disjoint copies of $G'_{\ell, j - 1}$, where $V(G_{\ell, j}) \coloneqq V(G'_{\ell, j - 1}) \times [2D]$ (so $V(G_{\ell, j}) \coloneqq V(G_\ell) \times [2D]^j$).  If $\ell \neq i$, then let $G'_{\ell, j} \coloneqq G_{\ell, j}$.  Otherwise, for each $v \in V(G_i)$, let $R_{j, v}$ be a $(D - d_{G_i, L}(v, c_j))$-regular graph where $V(R_{j, v}) \coloneqq \{v\} \times [2D]^j$ satisfying $N_{R_{j,v}}((v, x_1, \dots, x_j)) \subseteq \{v\} \times \{x_1\} \times \cdots \times \{x_{j-1}\}\times [2D]$ for every $(v, x_1, \dots, x_j) \in V(R_{j, v})$ (for example by taking $(2D)^{j - 1}$ disjoint copies of a bipartite graph with parts of size $D$).  Note that $R_{j, v}$ and $G_{i, j}$ are edge-disjoint, and let $G'_{i, j} \coloneqq G_{i, j} \cup \bigcup_{v \in V(G_i)}R_{j, v}$.  Let $G'_\ell \coloneqq G'_{\ell, k}$ for each $\ell \in [m]$, and note that \ref{embedding:containment} holds.

    Let $G' \coloneqq \bigcup_{\ell=1}^m G'_\ell$, and let $L'$ be the list assignment for $G'$ defined as follows.  For each $v \in V(G)$ and $(x_1, \dots, x_k) \in [2D]^k$, let $L'((v, x_1, \dots, x_k)) \coloneqq \{(c_j, x_1, \dots, x_{j - 1}, 0, x_{j + 1}, \dots, x_k) : c_j \in L(v)\}$.  Note that \ref{embedding:list-size} holds.  Now for every $(v, x_1, \dots, x_k) \in V(G_i) \times [2D]^k$ and $j \in [k]$ with $c_j \in L(v)$, we have
    \begin{multline*}
      d_{G'_i, L'}\left((v, x_1, \dots, x_k), (c_j, x_1, \dots, x_{j-1}, 0, x_{j+1}, \dots, x_k)\right) = \\
      d_{G_i, L}(v, c_j)
      + d_{R_{j, v}}((v, x_1, \dots, x_j)) = D,      
    \end{multline*}
    and for every $(v, x_1, \dots, x_k) \in V(G_\ell) \times [2D]^k$ for $\ell \in [m]\setminus\{i\}$ and every $j \in [k]$ with $c_j \in L(v)$, we have
    \begin{equation*}
      d_{G'_\ell, L'}\left((v, x_1, \dots, x_k), (c_j, x_1, \dots, x_{j-1}, 0, x_{j+1}, \dots, x_k)\right) =
      d_{G, L}(v, c_j).
    \end{equation*}
    In particular, \ref{embedding:degree} holds for every $\ell \in [i]$ by induction, as desired.  Finally, we complete the proof by relabelling vertices and colors so that \ref{embedding:embedding} and \ref{embedding:lists-preserved} also hold.
  \end{proof}

}
\end{document}